\documentclass[11pt, fleqn]{article}

\usepackage{amssymb, amsthm, amsmath}

\numberwithin{equation}{section}

\theoremstyle{plain}
\newtheorem{theorem}{Theorem}[section]
\newtheorem{lemma}{Lemma}[section]
\newtheorem{corollary}{Corollary}[section]

\theoremstyle{definition}
\newtheorem{definition}{Definition}[section]

\theoremstyle{remark}
\newtheorem{remark}{Remark}[section]

\title{The initial value problem for motion of incompressible viscous and heat-conductive fluids in Banach spaces}

\author{Ry\^{o}hei Kakizawa\thanks{Graduate School of Mathematical Sciences, The University of Tokyo, 3-8-1 Komaba Meguro-ku Tokyo 153-8914, Japan (\textit{E-mail address:} kakizawa@ms.u-tokyo.ac.jp)}}

\date{}

\begin{document}

\maketitle

\begin{abstract}
We consider the abstract initial value problem for the system of evolution equations which describe motion of incompressible viscous and heat-conductive fluids in a bounded domain.
It is difficulty of our problem that we do not neglect the viscous dissipation function in contrast to the Boussinesq approximation.
This problem has uniquely a mild solution locally in time for general initial data, and globally in time for small initial data.
Moreover, a mild solution of this problem can be a strong or classical solution under appropriate assumptions for initial data.
We prove the above properties by the theory of analytic semigroups on Banach spaces.
\end{abstract}


\section{Introduction}
Let $\Omega$ be a bounded domain in $\mathbb{R}^{n}$ $(n \in \mathbb{Z}, \ n\geq 2)$ with its $C^{2,1}$-boundary $\partial\Omega$.
Motion of incompressible viscous and heat-conductive fluids in $\Omega$ is described by the system of $n+2$ equations as follows:
\begin{equation}
\begin{cases}
\mathrm{div}u=0 & \mathrm{in} \ \Omega\times(0,T), \\
\rho\{\partial_{t}+(u\cdot\nabla)\}u=\rho f(\theta)-\nabla p+\mu\Delta u & \mathrm{in} \ \Omega\times(0,T), \\
\rho c_{v}\{\partial_{t}+(u\cdot\nabla)\}\theta=\Phi(u)+\kappa\Delta\theta & \mathrm{in} \ \Omega\times(0,T),
\end{cases}
\end{equation}
where $u=(u_{1}, \cdots, u_{n})$ is the fluid velocity, $p$ is the pressure, $\theta$ is the absolute temperature, $\rho$ is the density, $\mu$ is the coefficient of viscosity, $\kappa$ is the coefficient of heat conductivity, $c_{v}$ is the specific heat at constant volume, $f=(f_{1}, \cdots, f_{n})$ is the external force field affected by $\theta$, $\Phi(u)$ is the viscous dissipation function defined as
\begin{equation*}
\Phi(u)=\Phi(u,u), \ \Phi(u,v)=2\mu D(u):D(v), \ D(u)=\frac{1}{2}\left(\nabla u+(\nabla u)^{T}\right),
\end{equation*}
$(\nabla u)^{T}$ is the transposed matrix of $\nabla u$.
These equations correspond to the law of conservation of mass, momentum and energy respectively.
Moreover, it is required that $\rho$, $\mu$, $\kappa$ and $c_{v}$ are positive constants.
See, for example, \cite{Lamb}, \cite{Serrin} on conservation laws of fluid motion and the derivation of the above equations.

It is well known in \cite{Boussinesq} that the Boussinesq approximation is a simplified model of motion of incompressible viscous and heat-conductive fluids.
There is no doubt that many investigations on the Boussinesq approximation have been carried out for one hundred years.
The Rayleigh-B\'{e}nard convection can be considered as a typical phenomenon valid for the Boussinesq approximation in the case where the Rayleigh number $Ra$ is slightly larger than the critical Rayleigh number $Ra_{c}$.
It is an important physical property that formation of the Rayleigh-B\'{e}nard convection is characterized as the B\'{e}nard cellular pattern.
On the other hand, the collapse of the B\'{e}nard cellular pattern will be caused by the relative increase in $Ra$ to $Ra_{c}$.
In the case where $Ra$ is sufficiently larger than $Ra_{c}$, the Boussinesq approximation does not seem appropriate due to its neglect of the viscous dissipation function.
It is quite natural to consider (1.1) from the hydrodynamical point of view.
Some problems related to (1.1) have been studied in recent years.
Kagei and Skowron \cite{Kagei 1} discussed the existence and uniqueness of solutions of the initial-boundary value problem for motion of micropolar fluids with heat conduction in $\mathbb{R}^{3}$.
Moreover, Kagei \cite{Kagei 2} considered global attractors for the initial-boundary value problem for (1.1) in $\mathbb{R}^{2}$.
{\L}ukaszewicz and Krzy\.{z}anowski \cite{Lukaszewicz} treated the initial-boundary value problem for (1.1) in $\mathbb{R}^{3}$ with moving boundaries.
However, initial data $(u_{0},\theta_{0})$ in $L^{p}_{\sigma}(\Omega)\times L^{q}(\Omega)$ $(1<p<\infty, \ 1<q<\infty)$ except for $p=q=2$ and classical solutions of the initial-boundary value problem for (1.1) in anisotropic H\"{o}lder spaces are not considered in their results, where $L^{p}_{\sigma}(\Omega)$ is the closed subspace of $(L^{p}(\Omega))^{n}$ defined as in section 2.
It is necessary to discuss the existence, uniqueness and regularity of solutions of the initial-boundary value problem for (1.1) with initial data $(u_{0},\theta_{0})$ in $L^{p}_{\sigma}(\Omega)\times L^{q}(\Omega)$.

In order to meet the above requirement, we study the initial-boundary value problem for (1.1) with the following initial-boundary data:
\begin{equation}
\begin{cases}
u|_{t=0}=u_{0} & \mathrm{in} \ \Omega, \\
u|_{\partial\Omega}=0 & \mathrm{on} \ \partial\Omega\times(0,T), \\
\theta|_{t=0}=\theta_{0} & \mathrm{in} \ \Omega, \\
\theta|_{\partial\Omega}=\theta_{s} & \mathrm{on} \ \partial\Omega\times(0,T),
\end{cases}
\end{equation}
where $\theta_{s}$ is the surface temperature on $\partial\Omega$ assumed to be a nonnegative constant.
When we treat initial data $(u_{0},\theta_{0})$ in $L^{p}_{\sigma}(\Omega)\times L^{q}(\Omega)$, it is useful to transform (1.1), (1.2) into the abstract initial value problem for the system of evolution equations in the same Banach space as above.
It is explained in section 2 that (1.1), (1.2) are rewritten as follows:
\begin{equation}\tag{I}
\begin{cases}
d_{t}u+A_{p}u=F(u,\theta) & \mathrm{in} \ (0,T), \\
d_{t}\theta+B_{q}\theta=G(u,\theta) & \mathrm{in} \ (0,T), \\
u(0)=u_{0}, \\
\theta(0)=\theta_{0},
\end{cases}
\end{equation}
where $A_{p}$ and $B_{q}$ are sectorial operators in $L^{p}_{\sigma}(\Omega)$ and $L^{q}(\Omega)$ respectively, $F(u,\theta)$ and $G(u,\theta)$ are nonlinear terms corresponding to $(1.1)_{2}$ and $(1.1)_{3}$ respectively.
It is well known in \cite[Chapter 3]{Henry}, \cite[Chapter 6]{Pazy} that we can consider not only strong solutions but also mild solutions of (1.1), (1.2).

We are concerned with the existence, uniqueness and regularity of mild solutions of (1.1), (1.2) in this paper.
This problem has uniquely a mild solution locally in time for general initial data, and globally in time for small initial data.
Moreover, a mild solution of this problem can be a strong or classical solution under appropriate assumptions for initial data.
We prove the above properties by the argument based on \cite{Fujita}, \cite{Giga 3}, \cite{Hishida}.
First, the existence of local mild solutions is obtained from the successive approximation method.
Second, global a priori estimates for mild solutions of (1.1), (1.2) give the existence of global mild solutions, and make the asymptotic behavior of global mild solutions clear.

This paper is organized as follows: In section 2, we define basic notation used in this paper and a strong and mild solution of (1.1), (1.2), and state our main results and some lemmas for them.
We prove the existence and uniqueness of mild solutions of (1.1), (1.2) in section 3.
The regularity of mild solutions of (1.1), (1.2) is discussed in sections 4 and 5.

\section{Preliminaries and main results}
\subsection{Function spaces}
Function spaces and basic notation which we use throughout this paper are introduced as follows: The norm in $L^{r}(\Omega)$ $(1\leq r\leq\infty)$ and the norm in $W^{k,r}(\Omega)$ (the Sobolev space, $k \in \mathbb{Z}$, $k\geq 0$) are denoted by $\|\cdot\|_{r}$ and $\|\cdot\|_{k,r}$ respectively, $W^{0,r}(\Omega)=L^{r}(\Omega)$, $\|\cdot\|_{0,r}=\|\cdot\|_{r}$.
$C^{\infty}_{0}(\Omega)$ is the set of all functions which are infinitely differentiable and have compact support in $\Omega$.
$W^{k,r}_{0}(\Omega)$ is the completion of $C^{\infty}_{0}(\Omega)$ in $W^{k,r}(\Omega)$.
Let us introduce solenoidal function spaces.
$C^{\infty}_{0,\sigma}(\Omega):=\{u \in (C^{\infty}_{0}(\Omega))^{n} \ ; \ \mathrm{div}u=0\}$.
$L^{p}_{\sigma}(\Omega)$ $(1<p<\infty)$ is the completion of $C^{\infty}_{0,\sigma}(\Omega)$ in $(L^{p}(\Omega))^{n}$.
It follows from \cite[Theorem 2]{Fujiwara} that $(L^{p}(\Omega))^{n}$ is decomposed into $(L^{p}(\Omega))^{n}=L^{p}_{\sigma}(\Omega)\oplus L^{p}_{\pi}(\Omega)$, where $L^{p}_{\pi}(\Omega)=\{\nabla p \ ; \ p \in W^{1,p}(\Omega)\}$.
Let $P_{p}$ be the projection of $(L^{p}(\Omega))^{n}$ onto $L^{p}_{\sigma}(\Omega)$.
$C^{k,\gamma}(\overline{\Omega})$ $(0<\gamma\leq 1)$ is the H\"{o}lder space defined as in \cite[1.26--1.29]{Adams}, $C^{k,0}(\overline{\Omega})=C^{k}(\overline{\Omega})$, $C^{0}(\overline{\Omega})=C(\overline{\Omega})$.

Let $I$ be an interval in $\mathbb{R}$, $X$ be a Banach space.
$C(I;X)$ is the set of all $X$-valued functions which are continuous in $I$.
$C_{b}(I;X)$ is the set of all $X$-valued functions which are bounded continuous in $I$.
$C^{k}(I;X)$ $(k \in \mathbb{Z}, \ k\geq 0)$ is the set of all $X$-valued function which are continuously differentiable up to the order $k$ in $I$, $C^{0}(I;X)=C(I;X)$.
In the case where $I$ is a bounded closed interval in $\mathbb{R}$, $C^{0,\gamma}(I;X)$ $(0<\gamma\leq 1)$ is the set of all $X$-valued function which are uniformly H\"{o}lder continuous with the exponent $\gamma$ on $I$.
If $I$ is not bounded or closed, $u \in C^{0,\gamma}(I;X)$ means that $u \in C^{0,\gamma}(I_{1};X)$ for any bounded closed interval $I_{1}$ contained in $I$.
$C^{k,\gamma}(I;X)$ is the set of all $X$-valued functions $u$ which $u \in C^{k}(I;X)$ and $d^{k}_{t}u \in C^{0,\gamma}(I;X)$, $C^{k,0}(I;X)=C^{k}(I;X)$.

$C_{b}(\mathbb{R};\mathbb{R}^{n})$ is the set of all $\mathbb{R}^{n}$-valued functions which are bounded continuous in $\mathbb{R}$.
$C^{k}(\mathbb{R};\mathbb{R}^{n})$ $(k \in \mathbb{Z}, \ k\geq 0)$ is the set of all $\mathbb{R}^{n}$-valued functions which are continuously differentiable up to the order $k$ in $\mathbb{R}$, $C^{0}(\mathbb{R};\mathbb{R}^{n})=C(\mathbb{R};\mathbb{R}^{n})$.
$C^{0,1}(\mathbb{R};\mathbb{R}^{n})$ is the set of all $\mathbb{R}^{n}$-valued functions which are uniformly Lipschitz continuous in $\mathbb{R}$.

\subsection{Stokes and Laplace operator}
For the sake of simplicity, we assume that $\rho=1$, $\mu=1$, $\kappa=1$, $c_{v}=1$ and $\theta_{s}=0$ throughout this paper.
Let us introduce two linear operators $A_{p}$ $(1<p<\infty)$ and $B_{q}$ $(1<q<\infty)$ which appeared in (I).
$B_{q}$ is the Laplace operator in $L^{q}(\Omega)$ with the zero Dirichlet boundary condition defined as $B_{q}=-\Delta$, $D(B_{q})=W^{2,q}(\Omega)\cap W^{1,q}_{0}(\Omega)$, where $D(B_{q})$ is the domain of $B_{q}$.
We introduce the Stokes operator $A_{p}$ in $L^{p}_{\sigma}(\Omega)$ by $A_{p}=-P_{p}\Delta$, $D(A_{p})=(D(B_{p}))^{n}\cap L^{p}_{\sigma}(\Omega)$.
It is well known in \cite[Theorems 2.5.2 and 7.3.6]{Pazy}, \cite[Theorem 1]{Giga 1} that $B_{q}$ and $A_{p}$ are sectorial operators in $L^{q}(\Omega)$ and $L^{p}_{\sigma}(\Omega)$ respectively.
Therefore, $-B_{q}$ generates an uniformly bounded analytic semigroup $\{e^{-tB_{q}}\}_{t\geq 0}$ on $L^{q}(\Omega)$, fractional powers $B^{\beta}_{q}$ of $B_{q}$ can be defined for any $\beta\geq 0$, $B^{0}_{q}=I_{q}$, where $I_{q}$ is the identity operator in $L^{q}(\Omega)$.
Similarly to $B_{q}$, an uniformly bounded analytic semigroup $\{e^{-tA_{p}}\}_{t\geq 0}$ on $L^{p}_{\sigma}(\Omega)$ is generated, fractional powers $A^{\alpha}_{p}$ of $A_{p}$ are defined for any $\alpha\geq 0$.
Moreover, it follows from \cite[Theorem 3]{Giga 2} that $D(A^{\alpha}_{p})$ is characterized as $D(A^{\alpha}_{p})=(D(B^{\alpha}_{p}))^{n}\cap L^{p}_{\sigma}(\Omega)$ for any $0\leq \alpha\leq 1$.
Let us introduce Banach spaces derived from $A^{\alpha}_{p}$ and $B^{\beta}_{q}$.
$X^{\alpha}_{p}$ ($Y^{\beta}_{q}$) is defined as $D(A^{\alpha}_{p})$ ($D(B^{\beta}_{q})$) with the norm $\|\cdot\|_{X^{\alpha}_{p}}=\|A^{\alpha}_{p}\cdot\|_{p}$ ($\|\cdot\|_{Y^{\beta}_{q}}=\|B^{\beta}_{q}\cdot\|_{q}$).
$\Lambda_{1}$ is the first eigenvalue of the Laplace operator with the zero Dirichlet boundary condition.

We state some lemmas concerning sectorial operators in Banach spaces.
See, for example, \cite[Chapter 1]{Henry}, \cite[Chapter 2]{Pazy} on the theory of analytic semigroups on Banach spaces and fractional powers of sectorial operators.
\begin{lemma}[8, Theorem 1.4.3]
Let $1<p<\infty$, $1<q<\infty$, $\alpha\geq 0$, $\beta\geq 0$, $0<\lambda<\Lambda_{1}$.
Then
\begin{equation}
\|A^{\alpha}_{p}e^{-tA_{p}}u\|_{p}\leq C_{A_{p},\alpha,\lambda}t^{-\alpha}e^{-\lambda t}\|u\|_{p},
\end{equation}
\begin{equation}
\|B^{\beta}_{q}e^{-tB_{q}}\theta\|_{q}\leq C_{B_{q},\beta,\lambda}t^{-\beta}e^{-\lambda t}\|\theta\|_{q}
\end{equation}
for any $u \in L^{p}_{\sigma}(\Omega)$, $\theta \in L^{q}(\Omega)$, where $C_{A_{p}, \alpha,\lambda}$ and $C_{B_{q}, \beta,\lambda}$ are positive constants.
\end{lemma}
\begin{lemma}[8, Theorem 1.4.3]
Let $1<p<\infty$, $1<q<\infty$, $0<\alpha\leq 1$, $0<\beta\leq 1$.
Then
\begin{equation}
\|(e^{-tA_{p}}-I_{p})u\|_{p}\leq C_{A_{p},\alpha}t^{\alpha}\|u\|_{X^{\alpha}_{p}},
\end{equation}
\begin{equation}
\|(e^{-tB_{q}}-I_{q})\theta\|_{q}\leq C_{B_{q},\beta}t^{\beta}\|\theta\|_{Y^{\beta}_{q}}
\end{equation}
for any $u \in X^{\alpha}_{p}$, $\theta \in Y^{\beta}_{q}$, where $C_{A_{p}, \alpha}$ and $C_{B_{q}, \beta}$ are positive constants.
\end{lemma}
\begin{lemma}[8, Exercise 1.4.10]
Let $1<p<\infty$, $1<q<\infty$, $0<\alpha\leq 1$, $0<\beta\leq 1$.
Then
\begin{equation}
\|e^{-tA_{p}}u\|_{X^{\alpha}_{p}}=o(t^{-\alpha}) \ \mathrm{as} \ t\rightarrow +0,
\end{equation}
\begin{equation}
\|e^{-tB_{q}}\theta\|_{Y^{\beta}_{q}}=o(t^{-\beta}) \ \mathrm{as} \ t\rightarrow +0
\end{equation}
for any $u \in L^{p}_{\sigma}(\Omega)$, $\theta \in L^{q}(\Omega)$.
\end{lemma}
\begin{lemma}[8, Theorem 1.6.1]
Let $1<p<\infty$, $1<q<\infty$, $0\leq\alpha\leq 1$, $0\leq\beta\leq 1$.
Then
\begin{equation}
X^{\alpha}_{p}\hookrightarrow (W^{k,r}(\Omega))^{n} \ \mathrm{if} \ \frac{1}{p}-\frac{2\alpha-k}{n}\leq\frac{1}{r}\leq\frac{1}{p},
\end{equation}
\begin{equation}
Y^{\beta}_{q}\hookrightarrow W^{k,r}(\Omega) \ \mathrm{if} \ \frac{1}{q}-\frac{2\beta-k}{n}\leq\frac{1}{r}\leq\frac{1}{q},
\end{equation}
where $\hookrightarrow$ is the continuous inclusion.
\end{lemma}

\subsection{Abstract initial value problem for (1.1), (1.2)}
Let $1<p<\infty$, $1<q<\infty$, $0\leq \alpha_{0}<1$, $0\leq \beta_{0}<1$, $u_{0} \in X^{\alpha_{0}}_{p}$, $\theta_{0} \in Y^{\beta_{0}}_{q}$.
Then we apply $P_{p}$ to $(1.1)_{2}$, and get the following abstract initial value problem:
\begin{equation}\tag{I}
\begin{cases}
d_{t}u+A_{p}u=F(u,\theta) & \mathrm{in} \ (0,T], \\
d_{t}\theta+B_{q}\theta=G(u,\theta) & \mathrm{in} \ (0,T], \\
u(0)=u_{0}, \\
\theta(0)=\theta_{0},
\end{cases}
\end{equation}
where
\begin{equation*}
F(u,\theta):=-P_{p}(u\cdot\nabla)u+P_{p}f(\theta),
\end{equation*}
\begin{equation*}
G(u,\theta):=-(u\cdot\nabla)\theta+\Phi(u).
\end{equation*}
In order to solve (I), first of all, we shall find a solution satisfying the following abstract integral equations related to (I):
\begin{equation}\tag{II}
\begin{cases}
u(t)=e^{-tA_{p}}u_{0}+\displaystyle\int^{t}_{0}e^{-(t-s)A_{p}}F(u,\theta)(s)ds, \\
\theta(t)=e^{-tB_{q}}\theta_{0}+\displaystyle\int^{t}_{0}e^{-(t-s)B_{q}}G(u,\theta)(s)ds
\end{cases}
\end{equation}
for any $0\leq t\leq T$.
Let us introduce a strong and mild solution of (1.1), (1.2) defined on $[0,T]$.
A strong and mild solution of (1.1), (1.2) defined on $[0,\infty)$ is similarly defined.
\begin{definition}
$(u,\theta)$ is called a strong solution of (1.1), (1.2) if it satisfies
\begin{equation*}
u \in C([0,T];X^{\alpha_{0}}_{p})\cap C((0,T];X^{1}_{p}), \ d_{t}u \in C((0,T];L^{p}_{\sigma}(\Omega)),
\end{equation*}
\begin{equation*}
\theta \in C([0,T];Y^{\beta_{0}}_{q})\cap C((0,T];Y^{1}_{q}), \ d_{t}\theta \in C((0,T];L^{q}(\Omega))
\end{equation*}
and (I).
\end{definition}
\begin{definition}
$(u,\theta)$ is called a mild solution of (1.1), (1.2) if it satisfies
\begin{equation*}
u \in C([0,T];X^{\alpha_{0}}_{p}),
\end{equation*}
\begin{equation*}
\theta \in C([0,T];Y^{\beta_{0}}_{q})
\end{equation*}
and (II).
\end{definition}

\subsection{Main results}
We will state our main results in this subsection.
It is sufficient for our main results to be assumed that $f \in C^{0,1}(\mathbb{R};\mathbb{R}^{n})$ with the Lipschitz constant $L_{f}$, $f(0)=0$, $p$, $q$, $\alpha_{0}$ and $\beta_{0}$ satisfy the following inequalities:
\begin{equation}
\max\left\{1, \frac{n}{3} \right\}<p<\infty, \ 1<q<\infty, \ \frac{1}{p}-\frac{1}{2q}<\frac{1}{n}, \ \frac{1}{q}-\frac{1}{p}<\frac{2}{n},
\end{equation}
\begin{equation}
\begin{split}
&\max\left\{0, \frac{n}{2p}-\frac{1}{2}\right\}\leq\alpha_{0}<1, \ 0\leq\beta_{0}<1, \\
&\alpha_{0}-\frac{\beta_{0}}{2}-\frac{n}{2}\left(\frac{1}{p}-\frac{1}{2q}\right)\geq 0, \ -1<\alpha_{0}-\beta_{0}-\frac{n}{2}\left(\frac{1}{p}-\frac{1}{q}\right)\leq 1.
\end{split}
\end{equation}
The first purpose of this paper is to study the existence and uniqueness of mild solutions of (1.1), (1.2).
We shall prove the following theorems:
\begin{theorem}
Let $f \in C^{0,1}(\mathbb{R};\mathbb{R}^{n})$ with the Lipschitz constant $L_{f}$, $f(0)=0$, $p$, $q$, $\alpha_{0}$ and $\beta_{0}$ satisfy $(2.9)$, $(2.10)$, $u_{0} \in X^{\alpha_{0}}_{p}$, $\theta_{0} \in Y^{\beta_{0}}_{q}$.
Then there exists a positive constant $T_{*}\leq T$ depending only on $n$, $\Omega$, $p$, $q$, $\alpha_{0}$, $\beta_{0}$, $u_{0}$, $\theta_{0}$, $L_{f}$ and $T$ such that $(1.1)$, $(1.2)$ has uniquely a mild solution $(u,\theta)$ on $[0,T_{*}]$ satisfying the following continuity properties and estimates:
\begin{itemize}
\item[\rm{(i)}]For any $\alpha_{0}\leq\alpha<1$, $\beta_{0}\leq\beta<1$, $0<t\leq T_{*}$,
\begin{equation*}
t^{\alpha-\alpha_{0}}u \in C([0,T_{*}];X^{\alpha}_{p}),
\end{equation*}
\begin{equation*}
t^{\beta-\beta_{0}}\theta \in C([0,T_{*}];Y^{\beta}_{q}),
\end{equation*}
\begin{equation}
\|u(t)\|_{X^{\alpha}_{p}}\leq Ct^{\alpha_{0}-\alpha}(\|u_{0}\|_{X^{\alpha_{0}}_{p}}+\|\theta_{0}\|_{Y^{\beta_{0}}_{q}}),
\end{equation}
\begin{equation}
\|\theta(t)\|_{Y^{\beta}_{q}}\leq Ct^{\beta_{0}-\beta}(\|u_{0}\|_{X^{\alpha_{0}}_{p}}+\|\theta_{0}\|_{Y^{\beta_{0}}_{q}}),
\end{equation}
where $C$ is a positive constant independent of $u$, $\theta$ and $t$.
\item[\rm{(ii)}]For any $\alpha_{0}<\alpha<1$, $\beta_{0}<\beta<1$,
\begin{equation}
\|u(t)\|_{X^{\alpha}_{p}}=o(t^{\alpha_{0}-\alpha}) \ \mathrm{as} \ t\rightarrow +0,
\end{equation}
\begin{equation}
\|\theta(t)\|_{Y^{\beta}_{q}}=o(t^{\beta_{0}-\beta}) \ \mathrm{as} \ t\rightarrow +0.
\end{equation}
\end{itemize}
\end{theorem}
\begin{theorem}
Let $f \in C^{0,1}(\mathbb{R};\mathbb{R}^{n})$ with the Lipschitz constant $L_{f}$, $f(0)=0$, $p$, $q$, $\alpha_{0}$ and $\beta_{0}$ satisfy $(2.9)$, $(2.10)$, $u_{0} \in X^{\alpha_{0}}_{p}$, $\theta_{0} \in Y^{\beta_{0}}_{q}$, $0<\lambda<\Lambda_{1}$.
Then there exists a positive constant $\varepsilon$ depending only on $n$, $\Omega$, $p$, $q$, $\alpha_{0}$, $\beta_{0}$, $L_{f}$ and $\lambda$ such that $(1.1)$, $(1.2)$ has uniquely a mild solution $(u,\theta)$ on $[0,\infty)$ satisfying the following continuity properties and estimates:
\begin{equation*}
t^{\alpha-\alpha_{0}}e^{\lambda t}u \in C_{b}([0,\infty);X^{\alpha}_{p}),
\end{equation*}
\begin{equation*}
t^{\beta-\beta_{0}}e^{\lambda t}\theta \in C_{b}([0,\infty);Y^{\beta}_{q}),
\end{equation*}
\begin{equation}
\|u(t)\|_{X^{\alpha}_{p}}\leq Ct^{\alpha_{0}-\alpha}e^{-\lambda t}(\|u_{0}\|_{X^{\alpha_{0}}_{p}}+\|\theta_{0}\|_{Y^{\beta_{0}}_{q}}),
\end{equation}
\begin{equation}
\|\theta(t)\|_{Y^{\beta}_{q}}\leq Ct^{\beta_{0}-\beta}e^{-\lambda t}(\|u_{0}\|_{X^{\alpha_{0}}_{p}}+\|\theta_{0}\|_{Y^{\beta_{0}}_{q}})
\end{equation}
for any $\alpha_{0}\leq\alpha<1$, $\beta_{0}\leq\beta<1$, $t>0$, where $C$ is a positive constant independent of $u$, $\theta$ and $t$ provided that
\begin{equation*}
\|u_{0}\|_{X^{\alpha_{0}}_{p}}+\|\theta_{0}\|_{Y^{\beta_{0}}_{q}}\leq\varepsilon.
\end{equation*}
\end{theorem}
The second purpose of this paper is to discuss the regularity of mild solutions of (1.1), (1.2).
As for the regularity of $(d_{t}u,d_{t}\theta)$, it will be required that $p$, $q$, $\alpha_{0}$ and $\beta_{0}$ satisfy the following inequalities:
\begin{equation}
\alpha_{0}\geq n\left(\frac{1}{p}-\frac{1}{2q}\right), \ \beta_{0}\geq \frac{n}{2p}-\frac{1}{2},
\end{equation}
\begin{equation}
2\alpha_{0}-\beta_{0}\geq \frac{n}{2p}-\frac{1}{2}.
\end{equation}
We shall prove the following theorems:
\begin{theorem}
If a mild solution $(u,\theta)$ of $(1.1)$, $(1.2)$ in Theorem $2.1$ is defined on $[0,T]$, then $(u,\theta)$ is a strong solution of $(1.1)$, $(1.2)$ on $[0,T]$ satisfying the following continuity properties and estimates:
\begin{itemize}
\item[\rm{(i)}]For some $0<\hat{\alpha}<1$, $0<\hat{\beta}<1$,
\begin{equation*}
u \in C^{0,\hat{\alpha}}((0,T];X^{1}_{p}), \ d_{t}u \in C^{0,\hat{\alpha}}((0,T];L^{p}_{\sigma}(\Omega)),
\end{equation*}
\begin{equation*}
\theta \in C^{0,\hat{\beta}}((0,T];Y^{1}_{q}), \ d_{t}\theta \in C^{0,\hat{\beta}}((0,T];L^{q}(\Omega)),
\end{equation*}
and for any $0<t\leq T$,
\begin{equation}
\|u(t)\|_{X^{1}_{p}}\leq Ct^{\alpha_{0}-1}(\|u_{0}\|_{X^{\alpha_{0}}_{p}}+\|\theta_{0}\|_{Y^{\beta_{0}}_{q}}),
\end{equation}
\begin{equation}
\|\theta(t)\|_{Y^{1}_{q}}\leq Ct^{\beta_{0}-1}(\|u_{0}\|_{X^{\alpha_{0}}_{p}}+\|\theta_{0}\|_{Y^{\beta_{0}}_{q}}),
\end{equation}
where $C$ is a positive constant independent of $u$, $\theta$ and $t$.
\item[\rm{(ii)}]For any $0<\hat{\alpha}<1$, $0<\hat{\beta}<1$, $0\leq\alpha<1$, $0\leq\beta<1$, $0<\tilde{\alpha}<1-\alpha$, $0<\tilde{\beta}<1-\beta$,
\begin{equation*}
u \in C^{0,\hat{\alpha}}((0,T];X^{1}_{p}), \ d_{t}u \in C^{0,\tilde{\alpha}}((0,T];X^{\alpha}_{p}),
\end{equation*}
\begin{equation*}
\theta \in C^{0,\hat{\beta}}((0,T];Y^{1}_{q}), \ d_{t}\theta \in C^{0,\tilde{\beta}}((0,T];Y^{\beta}_{q})
\end{equation*}
provided that $p$, $q$, $\alpha_{0}$ and $\beta_{0}$ satisfy $(2.17)$.
\item[\rm{(iii)}]For any $0\leq\alpha<1$, $0\leq\beta<1$, $0<t\leq T$,
\begin{equation}
\|d_{t}u(t)\|_{X^{\alpha}_{p}}\leq Ct^{\alpha_{0}-\alpha-1}(\|u_{0}\|_{X^{\alpha_{0}}_{p}}+\|\theta_{0}\|_{Y^{\beta_{0}}_{q}}),
\end{equation}
\begin{equation}
\|d_{t}\theta(t)\|_{Y^{\beta}_{q}}\leq Ct^{\beta_{0}-\beta-1}(\|u_{0}\|_{X^{\alpha_{0}}_{p}}+\|\theta_{0}\|_{Y^{\beta_{0}}_{q}}),
\end{equation}
where $C$ is a positive constant independent of $u$, $\theta$ and $t$ provided that $p$, $q$, $\alpha_{0}$ and $\beta_{0}$ satisfy $(2.17)$, $(2.18)$.
\end{itemize}
\end{theorem}
\begin{theorem}
Let $(u,\theta)$ be a mild solution of $(1.1)$, $(1.2)$ on $[0,\infty)$ satisfying continuity properties and estimates $(2.15)$, $(2.16)$ in Theorem $2.2$.
Then $(u,\theta)$ is a strong solution of $(1.1)$, $(1.2)$ on $[0,\infty)$ satisfying the following continuity properties and estimates:
\begin{itemize}
\item[\rm{(i)}]For some $0<\hat{\alpha}<1$, $0<\hat{\beta}<1$,
\begin{equation*}
u \in C^{0,\hat{\alpha}}((0,\infty);X^{1}_{p}), \ d_{t}u \in C^{0,\hat{\alpha}}((0,\infty);L^{p}_{\sigma}(\Omega)),
\end{equation*}
\begin{equation*}
\theta \in C^{0,\hat{\beta}}((0,\infty);Y^{1}_{q}), \ d_{t}\theta \in C^{0,\hat{\beta}}((0,\infty);L^{q}(\Omega)),
\end{equation*}
and for any $t>0$,
\begin{equation}
\|u(t)\|_{X^{1}_{p}}\leq Ct^{\alpha_{0}-1}e^{-\lambda t}(\|u_{0}\|_{X^{\alpha_{0}}_{p}}+\|\theta_{0}\|_{Y^{\beta_{0}}_{q}}),
\end{equation}
\begin{equation}
\|\theta(t)\|_{Y^{1}_{q}}\leq Ct^{\beta_{0}-1}e^{-\lambda t}(\|u_{0}\|_{X^{\alpha_{0}}_{p}}+\|\theta_{0}\|_{Y^{\beta_{0}}_{q}}),
\end{equation}
where $C$ is a positive constant independent of $u$, $\theta$ and $t$.
\item[\rm{(ii)}]For any $0<\hat{\alpha}<1$, $0<\hat{\beta}<1$, $0\leq\alpha<1$, $0\leq\beta<1$, $0<\tilde{\alpha}<1-\alpha$, $0<\tilde{\beta}<1-\beta$,
\begin{equation*}
u \in C^{0,\hat{\alpha}}((0,\infty);X^{1}_{p}), \ d_{t}u \in C^{0,\tilde{\alpha}}((0,\infty);X^{\alpha}_{p}),
\end{equation*}
\begin{equation*}
\theta \in C^{0,\hat{\beta}}((0,\infty);Y^{1}_{q}), \ d_{t}\theta \in C^{0,\tilde{\beta}}((0,\infty);Y^{\beta}_{q})
\end{equation*}
provided that $p$, $q$, $\alpha_{0}$ and $\beta_{0}$ satisfy $(2.17)$.
\item[\rm{(iii)}]For any $0\leq\alpha<1$, $0\leq\beta<1$, $t>0$,
\begin{equation}
\|d_{t}u(t)\|_{X^{\alpha}_{p}}\leq Ct^{\alpha_{0}-\alpha-1}e^{-\lambda t}(\|u_{0}\|_{X^{\alpha_{0}}_{p}}+\|\theta_{0}\|_{Y^{\beta_{0}}_{q}}),
\end{equation}
\begin{equation}
\|d_{t}\theta(t)\|_{Y^{\beta}_{q}}\leq Ct^{\beta_{0}-\beta-1}e^{-\lambda t}(\|u_{0}\|_{X^{\alpha_{0}}_{p}}+\|\theta_{0}\|_{Y^{\beta_{0}}_{q}}),
\end{equation}
where $C$ is a positive constant independent of $u$, $\theta$ and $t$ provided that $p$, $q$, $\alpha_{0}$ and $\beta_{0}$ satisfy $(2.17)$, $(2.18)$.
\end{itemize}
\end{theorem}
Some detailed considerations admit that a strong solution of (1.1), (1.2) with initial data $(u_{0},\theta_{0}) \in L^{p}_{\sigma}(\Omega)\times L^{q}(\Omega)$ can be grasped in the classical sense.
Let $p$ and $q$ satisfy the following inequalities:
\begin{equation}
n<p<\infty, \ n<q<\infty, \ \frac{1}{p}-\frac{1}{2q}\leq 0, \ \frac{1}{q}-\frac{1}{p}<\frac{2}{n}.
\end{equation}
Then we can take $\alpha_{0}$ and $\beta_{0}$ in (2.10), (2.17), (2.18) as zeros.
It is derived from Theorems 2.3 and 2.4 that we obtain the following corollaries:
\begin{corollary}
Let $f \in C^{0,1}(\mathbb{R};\mathbb{R}^{n})\cap C^{1}(\mathbb{R};\mathbb{R}^{n})$, $f(0)=0$, $p$ and $q$ satisfy $(2.27)$, $u_{0} \in L^{p}_{\sigma}(\Omega)$, $\theta_{0} \in L^{q}(\Omega)$.
Then a strong solution $(u,\theta)$ of $(1.1)$, $(1.2)$ in Theorem $2.3$ is a classical solution of $(1.1)$, $(1.2)$ in $(0,T]$ satisfying the following continuity properties:
\begin{equation*}
u \in C^{0,\hat{\alpha}}((0,T];(C^{2,\alpha}(\overline{\Omega}))^{n}), \ d_{t}u \in C^{0,\tilde{\alpha}}((0,T];(C^{1,\alpha}(\overline{\Omega}))^{n}),
\end{equation*}
\begin{equation*}
\theta \in C^{0,\hat{\beta}}((0,T];C^{2,\beta}(\overline{\Omega})), \ d_{t}\theta \in C^{0,\tilde{\beta}}((0,T];C^{1,\beta}(\overline{\Omega}))
\end{equation*}
for any $0<\hat{\alpha}<1/2$, $0<\hat{\beta}<1/2$, $0<\alpha<1-n/p$, $0<\beta<1-n/q$ and for some $0<\tilde{\alpha}<1$, $0<\tilde{\beta}<1$.
\end{corollary}
\begin{corollary}
Let $f \in C^{0,1}(\mathbb{R};\mathbb{R}^{n})\cap C^{1}(\mathbb{R};\mathbb{R}^{n})$, $f(0)=0$, $p$ and $q$ satisfy $(2.27)$, $u_{0} \in L^{p}_{\sigma}(\Omega)$, $\theta_{0} \in L^{q}(\Omega)$.
Then a strong solution $(u,\theta)$ of $(1.1)$, $(1.2)$ in Theorem $2.4$ is a classical solution of $(1.1)$, $(1.2)$ in $(0,\infty)$ satisfying the following continuity properties:
\begin{equation*}
u \in C^{0,\hat{\alpha}}((0,\infty);(C^{2,\alpha}(\overline{\Omega}))^{n}), \ d_{t}u \in C^{0,\tilde{\alpha}}((0,\infty);(C^{1,\alpha}(\overline{\Omega}))^{n}),
\end{equation*}
\begin{equation*}
\theta \in C^{0,\hat{\beta}}((0,\infty);C^{2,\beta}(\overline{\Omega})), \ d_{t}\theta \in C^{0,\tilde{\beta}}((0,\infty);C^{1,\beta}(\overline{\Omega}))
\end{equation*}
for any $0<\hat{\alpha}<1/2$, $0<\hat{\beta}<1/2$, $0<\alpha<1-n/p$, $0<\beta<1-n/q$ and for some $0<\tilde{\alpha}<1$, $0<\tilde{\beta}<1$.
\end{corollary}
\begin{remark}
It can be easily seen from \cite[Theorem 7.3.6]{Pazy}, \cite[Theorem 1.3]{Shibata} that our main results are still valid, instead of $(1.2)$, for the following initial-boundary data:
\begin{equation*}
\begin{cases}
u|_{t=0}=u_{0} & \mathrm{in} \ \Omega, \\
u_{\nu}|_{\partial\Omega}=0 & \mathrm{on} \ \partial\Omega\times(0,T), \\
K(T(u,p)\nu)_{\tau}+(1-K)u_{\tau}|_{\partial\Omega}=0 & \mathrm{on} \ \partial\Omega\times(0,T), \\
\theta|_{t=0}=\theta_{0} & \mathrm{in} \ \Omega, \\
\theta|_{\partial\Omega}=\theta_{s} \ \mathrm{or} \ \kappa\partial_{\nu}\theta+\kappa_{s}\theta|_{\partial\Omega}=0 & \mathrm{on} \ \partial\Omega\times(0,T),
\end{cases}
\end{equation*}
where $\nu$ is the outward unit normal vector on $\partial\Omega$, $u_{\nu}:=\nu\cdot u$, $u_{\tau}:=u-u_{\nu}\nu$, $T(u,p)$ is the Cauchy stress tensor defined as
\begin{equation*}
T(u,p)=-pI_{n}+2\mu D(u),
\end{equation*}
$I_{n}$ is the $n$-th identity matrix, $0\leq K<1$ is a constant, $\kappa_{s}$ is a positive constant.
Moreover, it is useful to remark that $(T(u,p)\nu)_{\tau}=T(u,p)\nu-(\nu\cdot T(u,p)\nu)\nu=2\mu(D(u)\nu)_{\tau}$.
\end{remark}

\subsection{$L^{p}_{\sigma}\times L^{q}$-estimates for nonlinear terms}
We will state and prove some lemmas which play an important role throughout this paper.
They allow us to obtain $L^{p}_{\sigma}$-estimates for $F(u,\theta)$ and $L^{q}$-estimates for $G(u,\theta)$.
\begin{lemma}[7, Lemma 2.2]
Let $1<p<\infty$,
\begin{equation*}
\alpha_{1}>0, \ 0\leq\delta_{1}<\frac{1}{2}+\frac{n}{2}\left(1-\frac{1}{p}\right), \ \alpha_{1}+\delta_{1}>\frac{1}{2}, \ 2\alpha_{1}+\delta_{1}\geq\frac{n}{2p}+\frac{1}{2}.
\end{equation*}
Then
\begin{equation}
\|A^{-\delta_{1}}_{p}P_{p}(u\cdot\nabla)v\|_{p}\leq C_{1}\|u\|_{X^{\alpha_{1}}_{p}}\|v\|_{X^{\alpha_{1}}_{p}}
\end{equation}
for any $u, v \in X^{\alpha_{1}}_{p}$, where $C_{1}=C_{1}(\alpha_{1},\delta_{1})$ is a positive constant.
\end{lemma}
\begin{lemma}[9, Lemma 3.3]
Let $1<p<\infty$, $1<q<\infty$,
\begin{equation*}
\begin{split}
&\alpha_{2}, \beta_{2}\geq 0, \ \alpha_{2}>\frac{n}{2}\left(\frac{1}{p}-\frac{1}{q}\right), \ 0\leq\delta_{2}<\frac{1}{2}+\frac{n}{2}\left(1-\frac{1}{q}\right), \ \beta_{2}+\delta_{2}>\frac{1}{2}, \\
&\alpha_{2}+\beta_{2}+\delta_{2}\geq\frac{n}{2p}+\frac{1}{2}.
\end{split}
\end{equation*}
Then
\begin{equation}
\|B^{-\delta_{2}}_{q}(u\cdot\nabla)\theta\|_{q}\leq C_{2}\|u\|_{X^{\alpha_{2}}_{p}}\|\theta\|_{Y^{\beta_{2}}_{q}}
\end{equation}
for any $u \in X^{\alpha_{2}}_{p}$, $\theta \in Y^{\beta_{2}}_{q}$,where $C_{2}=C_{2}(\alpha_{2},\beta_{2},\delta_{2})$ is a positive constant.
\end{lemma}
\begin{lemma}
Let $1<p<\infty$, $1<q<\infty$,
\begin{equation*}
\max\left\{0, \frac{1}{2}+\frac{n}{2}\left(\frac{1}{p}-\frac{1}{2q}\right)\right\}\leq\alpha_{2}\leq 1.
\end{equation*}
Then
\begin{equation}
\|\Phi(u,v)\|_{q}\leq C_{3}\|u\|_{X^{\alpha_{2}}_{p}}\|v\|_{X^{\alpha_{2}}_{p}}
\end{equation}
for any $u, v \in X^{\alpha_{2}}_{p}$, where $C_{3}=C_{3}(\alpha_{2})$ is a positive constant.
\end{lemma}
\begin{proof}
After applying the Schwarz inequality to $\|\Phi(u,v)\|_{q}$, we can obtain (2.30) by (2.7) with $\alpha=\alpha_{2}$, $k=1$ and $r=2q$.
\end{proof}
\begin{lemma}
Let $f \in C^{0,1}(\mathbb{R};\mathbb{R}^{n})$ with the Lipschitz constant $L_{f}$, $f(0)=0$, $1<p<\infty$, $1<q<\infty$,
\begin{equation*}
\max\left\{0, \ \frac{n}{2}\left(\frac{1}{q}-\frac{1}{p}\right) \right\}\leq\beta_{1}\leq 1.
\end{equation*}
Then
\begin{equation}
\|P_{p}f(\theta)\|_{p}\leq C_{4}L_{f}\|\theta\|_{Y^{\beta_{1}}_{q}}
\end{equation}
for any $\theta \in Y^{\beta_{1}}_{q}$, where $C_{4}=C_{4}(\beta_{1})$ is a positive constant.
\end{lemma}
\begin{proof}
It is known in \cite[Theorem 1]{Fujiwara} that $P_{p}$ is a bounded operator in $(L^{p}(\Omega))^{n}$.
Since $\|f(\theta)\|_{p}\leq L_{f}\|\theta\|_{p}$, (2.31) follows from (2.8) with $\beta=\beta_{1}$, $k=0$ and $r=p$.
\end{proof}

\subsection{$X^{\alpha}_{p}\times Y^{\beta}_{q}$-estimates for nonlinear terms}
First, we will fix four exponents $\alpha_{1}$, $\alpha_{2}$, $\beta_{1}$ and $\beta_{2}$ in Lemmas 2.5--2.8 after the choice of two exponents $\delta_{1}$ in Lemma 2.5 and $\delta_{2}$ in Lemma 2.6.
We take $\delta_{1}$ ($\delta_{2}$) as zero in the case where $\alpha_{0}>0$ ($\beta_{0}>0$), and as an arbitrary positive constant in the case where $\alpha_{0}=0$ ($\beta_{0}=0$).
It is essential for (2.9), (2.10) that we make an appropriate choice of $\alpha_{1}$ in Lemma 2.5, $\alpha_{2}$ in Lemmas 2.6 and 2.7, $\beta_{1}$ in Lemma 2.8 and $\beta_{2}$ in Lemma 2.6.
Some elementary demonstrations admit that we can chose $\alpha_{0}<\alpha_{1}<1-\delta_{1}$, $\alpha_{0}<\alpha_{2}<1-\delta_{1}$, $\beta_{0}<\beta_{1}<1-\delta_{2}$ and $\beta_{0}<\beta_{2}<1-\delta_{2}$ which satisfy not only assumptions for Lemmas 2.5--2.8 but also
\begin{equation}
2\alpha_{1}+\delta_{1}\leq 1+\alpha_{0}, \ \alpha_{2}+\beta_{2}+\delta_{2}\leq 1+\alpha_{0}, \ \alpha_{2}\leq \alpha_{0}+\frac{1-\beta_{0}}{2},
\end{equation}
\begin{equation}
\begin{cases}
\beta_{1}<1-\alpha_{0}+\beta_{0} & \mathrm{if} \ \left|\alpha_{0}-\beta_{0}-\dfrac{n}{2}\left(\dfrac{1}{p}-\dfrac{1}{q}\right)\right|<1, \\
\beta_{1}=1-\alpha_{0}+\beta_{0} & \mathrm{if} \ \alpha_{0}-\beta_{0}-\dfrac{n}{2}\left(\dfrac{1}{p}-\dfrac{1}{q}\right)=1.
\end{cases}
\end{equation}
These exponents are fixed throughout this paper.

Second, we obtain $X^{\alpha}_{p}\times Y^{\beta}_{q}$-estimates for nonlinear terms which appeared in (II).
Let $0\leq \alpha<1-\delta_{1}$, $0\leq \beta<1-\delta_{2}$, $0<\lambda<\Lambda_{1}$, and set
\begin{equation*}
\mathcal{F}(u,\theta)(t)=\int^{t}_{0}e^{-(t-s)A_{p}}F(u,\theta)(s)ds,
\end{equation*}
\begin{equation*}
\mathcal{G}(u,\theta)(t)=\int^{t}_{0}e^{-(t-s)B_{q}}G(u,\theta)(s)ds.
\end{equation*}
Then $\|\mathcal{F}(u,\theta)(t)\|_{X^{\alpha}_{p}}$ and $\|\mathcal{G}(u,\theta)(t)\|_{Y^{\beta}_{q}}$ are bounded as follows:
\begin{equation}
\begin{split}
\|\mathcal{F}(u,\theta)&(t)\|_{X^{\alpha}_{p}} \\
\leq& C_{A_{p},\alpha+\delta_{1},\lambda}C_{1}\int^{t}_{0}(t-s)^{-(\alpha+\delta_{1})}e^{-\lambda(t-s)}\|u(s)\|^{2}_{X^{\alpha_{1}}_{p}}ds \\
&+C_{A_{p},\alpha,\lambda}C_{4}L_{f}\int^{t}_{0}(t-s)^{-\alpha}e^{-\lambda(t-s)}\|\theta(s)\|_{Y^{\beta_{1}}_{q}}ds,
\end{split}
\end{equation}
\begin{equation}
\begin{split}
\|\mathcal{G}(u,\theta)&(t)\|_{Y^{\beta}_{q}} \\
\leq& C_{B_{q},\beta+\delta_{2},\lambda}C_{2}\int^{t}_{0}(t-s)^{-(\beta+\delta_{2})}e^{-\lambda(t-s)}\|u(s)\|_{X^{\alpha_{2}}_{p}}\|\theta(s)\|_{Y^{\beta_{2}}_{q}}ds \\
&+C_{B_{q},\beta,\lambda}C_{3}\int^{t}_{0}(t-s)^{-\beta}e^{-\lambda(t-s)}\|u(s)\|^{2}_{X^{\alpha_{2}}_{p}}ds
\end{split}
\end{equation}
for any $0\leq t\leq T$.
Let $(u_{1},\theta_{1})$ and $(u_{2},\theta_{2})$ be two mild solutions of (1.1), (1.2).
Then we have the following inequalities:
\begin{equation}
\begin{split}
\|\mathcal{F}(u_{2},\theta_{2})&(t)-\mathcal{F}(u_{1},\theta_{1})(t)\|_{X^{\alpha}_{p}} \\
\leq& C_{A_{p},\alpha+\delta_{1},\lambda}C_{1}\int^{t}_{0}(t-s)^{-(\alpha+\delta_{1})}e^{-\lambda(t-s)}(\|u_{1}(s)\|_{X^{\alpha_{1}}_{p}}+\|u_{2}(s)\|_{X^{\alpha_{1}}_{p}}) \\
&\times \|(u_{2}-u_{1})(s)\|_{X^{\alpha_{1}}_{p}}ds \\
&+C_{A_{p},\alpha,\lambda}C_{4}L_{f}\int^{t}_{0}(t-s)^{-\alpha}e^{-\lambda(t-s)}\|(\theta_{2}-\theta_{1})(s)\|_{Y^{\beta_{1}}_{q}}ds,
\end{split}
\end{equation}
\begin{equation}
\begin{split}
\|\mathcal{G}(u_{2},\theta_{2})&(t)-\mathcal{G}(u_{1},\theta_{1})(t)\|_{Y^{\beta}_{q}} \\
\leq& C_{B_{q},\beta+\delta_{2},\lambda}C_{2}\int^{t}_{0}(t-s)^{-(\beta+\delta_{2})}e^{-\lambda(t-s)}\|(u_{2}-u_{1})(s)\|_{X^{\alpha_{2}}_{p}} \\
&\times\|\theta_{2}(s)\|_{Y^{\beta_{2}}_{q}}ds \\
&+C_{B_{q},\beta+\delta_{2},\lambda}C_{2}\int^{t}_{0}(t-s)^{-(\beta+\delta_{2})}e^{-\lambda(t-s)}\|u_{1}(s)\|_{X^{\alpha_{2}}_{p}} \\
&\times\|(\theta_{2}-\theta_{1})(s)\|_{Y^{\beta_{2}}_{q}}ds \\
&+C_{B_{q},\beta,\lambda}C_{3}\int^{t}_{0}(t-s)^{-\beta}e^{-\lambda(t-s)}(\|u_{1}(s)\|_{X^{\alpha_{2}}_{p}}+\|u_{2}(s)\|_{X^{\alpha_{2}}_{p}}) \\
&\times \|(u_{2}-u_{1})(s)\|_{X^{\alpha_{2}}_{p}}ds
\end{split}
\end{equation}
for any $0\leq t\leq T$.

\section{Proof of Theorems 2.1 and 2.2}
We will prove Theorems 2.1 and 2.2 in this section.
In proving our main results, simplified notation is given as follows: We drop two subscripts $p$ and $q$ attached to $P$, $A$, $B$, $X^{\alpha}$ and $Y^{\beta}$ in the sequel.
It is useful to remark that a generic positive constant independent of $u$, $\theta$ and $t$ is simply denoted by $C$.

\subsection{Existence of local mild solutions}
We construct a mild solution $(u,\theta)$ of (1.1), (1.2) by the following successive approximation $(u^{m},\theta^{m})$ $(m \in \mathbb{Z}, \ m\geq 0)$:
\begin{equation}
\begin{cases}
u^{0}(t)=e^{-tA}u_{0}, \\
\theta^{0}(t)=e^{-tB}\theta_{0},
\end{cases}
\end{equation}
\begin{equation}
\begin{cases}
u^{m+1}=u^{0}+\mathcal{F}(u^{m},\theta^{m}), \\
\theta^{m+1}=\theta^{0}+\mathcal{G}(u^{m},\theta^{m}).
\end{cases}
\end{equation}
It is assured by the following lemma that $\{t^{\alpha-\alpha_{0}}u^{m}\}_{m}$ and $\{t^{\beta-\beta_{0}}\theta^{m}\}_{m}$ are well-defined as sequences in $C([0,T];X^{\alpha})$ for $\alpha=\alpha_{1}, \alpha_{2}$ and in $C([0,T];Y^{\beta})$ for $\beta=\beta_{1}, \beta_{2}$ respectively.
\begin{lemma}
Let $\alpha=\alpha_{1}, \alpha_{2}$, $\beta=\beta_{1}, \beta_{2}$.
Then there exist monotone increasing continuous functions $K^{m}_{1,\alpha}$ and $K^{m}_{2,\beta}$ on $[0,T]$ for any $m \in \mathbb{Z}$, $m\geq 0$ such that $K^{m}_{1,\alpha}(0)=0$, $K^{m}_{2,\beta}(0)=0$,
\begin{equation}
\|u^{m}(t)\|_{X^{\alpha}}\leq K^{m}_{1,\alpha}(t)t^{\alpha_{0}-\alpha},
\end{equation}
\begin{equation}
\|\theta^{m}(t)\|_{Y^{\beta}}\leq K^{m}_{2,\beta}(t)t^{\beta_{0}-\beta}
\end{equation}
for any $0<t\leq T$, $K^{m}_{1,\alpha}\leq K^{m+1}_{1,\alpha}$, $K^{m}_{2,\beta}\leq K^{m+1}_{2,\beta}$ on $[0,T]$.
\end{lemma}
\begin{proof}
We give the inductive definition of $K^{m}_{1,\alpha}$ and $K^{m}_{2,\beta}$ with respect to $m$.
$K^{0}_{1,\alpha}$ and $K^{0}_{2,\beta}$ are defined as
\begin{equation}
K^{0}_{1,\alpha}(t)=\sup_{0<s\leq t}s^{\alpha-\alpha_{0}}\|u^{0}(s)\|_{X^{\alpha}},
\end{equation}
\begin{equation}
K^{0}_{2,\beta}(t)=\sup_{0<s\leq t}s^{\beta-\beta_{0}}\|\theta^{0}(s)\|_{Y^{\beta}}.
\end{equation}
It is obvious from (3.5), (3.6) that (3.3), (3.4) with $m=0$ hold for any $0<t\leq T$.
Moreover, (2.5), (2.6) yield that $K^{0}_{1,\alpha}(0)=0$, $K^{0}_{2,\beta}(0)=0$.
Assume that there exist $K^{m}_{1,\alpha}$ and $K^{m}_{2,\beta}$ for some $m \in \mathbb{Z}$, $m\geq 0$.
After applying (2.34), (2.35) to $(u^{m},\theta^{m})$, it is derived from (3.2) that we have the following inequalities:
\begin{equation*}
\begin{split}
\|u^{m+1}(t)\|_{X^{\alpha}}\leq& K^{0}_{1,\alpha}(t)t^{\alpha_{0}-\alpha} \\
&+C_{A,\alpha+\delta_{1},\lambda}C_{1}B(1-(\alpha+\delta_{1}),1+2(\alpha_{0}-\alpha_{1})) \\
&\times K^{m}_{1,\alpha_{1}}(t)^{2}t^{1+2\alpha_{0}-\alpha-2\alpha_{1}-\delta_{1}} \\
&+C_{A,\alpha,\lambda}C_{4}L_{f}B(1-\alpha,1+\beta_{0}-\beta_{1})K^{m}_{2,\beta_{1}}(t)t^{1+\beta_{0}-\alpha-\beta_{1}},
\end{split}
\end{equation*}
\begin{equation*}
\begin{split}
\|\theta^{m+1}(t)\|_{Y^{\beta}}\leq& K^{0}_{2,\beta}(t)t^{\beta_{0}-\beta} \\
&+C_{B,\beta+\delta_{2},\lambda}C_{2}B(1-(\beta+\delta_{2}),1+\alpha_{0}+\beta_{0}-\alpha_{2}-\beta_{2}) \\
&\times K^{m}_{1,\alpha_{2}}(t)K^{m}_{2,\beta_{2}}(t)t^{1+\alpha_{0}+\beta_{0}-\beta-\alpha_{2}-\beta_{2}-\delta_{2}} \\
&+C_{B,\beta,\lambda}C_{3}B(1-\beta,1+2(\alpha_{0}-\alpha_{2}))K^{m}_{1,\alpha_{2}}(t)^{2}t^{1+2\alpha_{0}-\beta-2\alpha_{2}}
\end{split}
\end{equation*}
for any $0<t\leq T$, where $B(x,y)$ is the beta function.
Therefore, $K^{m+1}_{1,\alpha}$ and $K^{m+1}_{2,\beta}$ can be defined as
\begin{equation}
\begin{split}
K^{m+1}_{1,\alpha}(t)=&K^{0}_{1,\alpha}(t) \\
&+C_{A,\alpha+\delta_{1},\lambda}C_{1}B(1-(\alpha+\delta_{1}),1+2(\alpha_{0}-\alpha_{1})) \\
&\times K^{m}_{1,\alpha_{1}}(t)^{2}t^{1+\alpha_{0}-2\alpha_{1}-\delta_{1}} \\
&+C_{A,\alpha,\lambda}C_{4}L_{f}B(1-\alpha,1+\beta_{0}-\beta_{1})K^{m}_{2,\beta_{1}}(t)t^{1+\beta_{0}-\alpha_{0}-\beta_{1}},
\end{split}
\end{equation}
\begin{equation}
\begin{split}
K^{m+1}_{2,\beta}(t)=&K^{0}_{2,\beta}(t) \\
&+C_{B,\beta+\delta_{2},\lambda}C_{2}B(1-(\beta+\delta_{2}),1+\alpha_{0}+\beta_{0}-\alpha_{2}-\beta_{2}) \\
&\times K^{m}_{1,\alpha_{2}}(t)K^{m}_{2,\beta_{2}}(t)t^{1+\alpha_{0}-\alpha_{2}-\beta_{2}-\delta_{2}} \\
&+C_{B,\beta,\lambda}C_{3}B(1-\beta,1+2(\alpha_{0}-\alpha_{2}))K^{m}_{1,\alpha_{2}}(t)^{2}t^{1+2\alpha_{0}-\beta_{0}-2\alpha_{2}}.
\end{split}
\end{equation}
It follows from (3.7), (3.8) that
\begin{equation*}
\|u^{m+1}(t)\|_{X^{\alpha}}\leq K^{m+1}_{1,\alpha}(t)t^{\alpha_{0}-\alpha},
\end{equation*}
\begin{equation*}
\|\theta^{m+1}(t)\|_{Y^{\beta}}\leq K^{m+1}_{2,\beta}(t)t^{\beta_{0}-\beta}
\end{equation*}
for any $0<t\leq T$.
Furthermore, we utilize inductive assumptions for $K^{m}_{1,\alpha}$ and $K^{m}_{2,\beta}$, and conclude that $K^{m+1}_{1,\alpha}(0)=0$, $K^{m+1}_{2,\beta}(0)=0$.
It can be easily seen from the induction with respect to $m$ that $K^{m}_{1,\alpha}\leq K^{m+1}_{1,\alpha}$, $K^{m}_{2,\beta}\leq K^{m+1}_{2,\beta}$ on $[0,T]$ for any $m \in \mathbb{Z}$, $m\geq 0$.
\end{proof}
We can see that a mild solution $(u,\theta)$ of (1.1), (1.2) is constructed by the following lemmas.
Set $K^{m}(t)=\max\{K^{m}_{1,\alpha}(t), K^{m}_{2,\beta}(t) \ ; \ \alpha=\alpha_{1}, \alpha_{2}, \beta=\beta_{1}, \beta_{2}\}$.
Then it follows from Lemma 3.1 that $K^{m}$ is a monotone increasing continuous function on $[0,T]$ satisfying $K^{m}(0)=0$, $K^{m}\leq K^{m+1}$ on $[0,T]$ for any $m \in \mathbb{Z}$, $m\geq 0$.
It is required that $C$ is independent of not only $u$, $\theta$ and $t$ but also $m$ throughout this subsection.
\begin{lemma}
Let $\alpha_{0}$ and $\beta_{0}$ satisfy $(2.33)_{1}$.
Then there exists a positive constant $T_{1}\leq T$ depending only on $n$, $\Omega$, $p$, $q$, $\alpha_{0}$, $\beta_{0}$, $u_{0}$, $\theta_{0}$, $L_{f}$ and $T$ such that $\{t^{\alpha-\alpha_{0}}u^{m}\}_{m}$ and $\{t^{\beta-\beta_{0}}\theta^{m}\}_{m}$ are Cauchy sequences in $C([0,T_{1}];X^{\alpha})$ for $\alpha=\alpha_{1}, \alpha_{2}$ and in $C([0,T_{1}];Y^{\beta})$ for $\beta=\beta_{1}, \beta_{2}$ respectively.
\end{lemma}
\begin{proof}
It follows from (2.32), $(2.33)_{1}$, (3.7), (3.8) that $K^{m}$ satisfies the following inductive inequality with respect to $m$:
\begin{equation}
K^{m+1}(t)\leq K^{0}(t)+CK^{m}(t)^{2}+CK^{m}(t)t^{1+\beta_{0}-\alpha_{0}-\beta_{1}}
\end{equation}
for any $0<t\leq T$, $m \in \mathbb{Z}$, $m\geq 0$.
Since $K^{0}(0)=0$, $K^{m}\leq K^{m+1}$ on $[0,T]$, $1+\beta_{0}-\alpha_{0}-\beta_{1}>0$, an elementary calculation shows that there exists a positive constant $\tau_{1}\leq T$ such that $K^{m}\leq CK^{0}$ on $[0,\tau_{1}]$.
Therefore, we can utilize (3.3), (3.4) to obtain the following inequalities:
\begin{equation}
\max_{\alpha=\alpha_{1}, \alpha_{2}}\{t^{\alpha-\alpha_{0}}\|u^{m}(t)\|_{X^{\alpha}}\}\leq CK^{0}(t),
\end{equation}
\begin{equation}
\max_{\beta=\beta_{1}, \beta_{2}}\{t^{\beta-\beta_{0}}\|\theta^{m}(t)\|_{Y^{\beta}}\}\leq CK^{0}(t)
\end{equation}
for any $0<t\leq \tau_{1}$.
It is sufficient for the conclusion that we give $X^{\alpha}$-estimates for $u^{m+1}-u^{m}$ and $Y^{\beta}$-estimates for $\theta^{m+1}-\theta^{m}$.
It can be easily seen from (2.34), (2.35), (3.2) with $m=0$ that
\begin{equation*}
\max_{\alpha=\alpha_{1}, \alpha_{2}}\{t^{\alpha-\alpha_{0}}\|(u^{1}-u^{0})(t)\|_{X^{\alpha}}\}\leq CK^{0}(t)(K^{0}(t)+1),
\end{equation*}
\begin{equation*}
\max_{\beta=\beta_{1}, \beta_{2}}\{t^{\beta-\beta_{0}}\|(\theta^{1}-\theta^{0})(t)\|_{Y^{\beta}}\}\leq CK^{0}(t)^{2}
\end{equation*}
for any $0<t\leq \tau_{1}$.
We utilize (2.36), (2.37), (3.10), (3.11) and the induction with respect to $m$, and obtain the following inequalities:
\begin{equation}
\begin{split}
\max_{\alpha=\alpha_{1}, \alpha_{2}}\{t^{\alpha-\alpha_{0}}\|(u^{m+1}-u^{m})(t)\|_{X^{\alpha}}\}\leq& CK^{0}(t)(K^{0}(t)+1) \\
&\times \{C(K^{0}(t)+t^{1+\beta_{0}-\alpha_{0}-\beta_{1}})\}^{m},
\end{split}
\end{equation}
\begin{equation}
\max_{\beta=\beta_{1}, \beta_{2}}\{t^{\beta-\beta_{0}}\|(\theta^{m+1}-\theta^{m})(t)\|_{Y^{\beta}}\}\leq CK^{0}(t)(K^{0}(t)+1)(CK^{0}(t))^{m}
\end{equation}
for any $0<t\leq \tau_{1}$.
Since $K^{0}(0)=0$, $1+\beta_{0}-\alpha_{0}-\beta_{1}>0$, we can take a positive constant $T_{1}\leq \tau_{1}$ satisfying $C(K^{0}(T_{1})+T^{1+\beta_{0}-\alpha_{0}-\beta_{1}}_{1})<1$, $CK^{0}(T_{1})<1$.
Then $\{t^{\alpha-\alpha_{0}}u^{m}\}_{m}$ and $\{t^{\beta-\beta_{0}}\theta^{m}\}_{m}$ are Cauchy sequences in $C([0,T_{1}];X^{\alpha})$ and in $C([0,T_{1}];Y^{\beta})$ respectively.
\end{proof}
\begin{lemma}
Let $\alpha_{0}$ and $\beta_{0}$ satisfy $(2.33)_{2}$.
Then there exists a positive constant $T_{2}\leq T$ depending only on $n$, $\Omega$, $p$, $q$, $\alpha_{0}$, $\beta_{0}$, $u_{0}$, $\theta_{0}$, $L_{f}$ and $T$ such that $\{t^{\alpha-\alpha_{0}}u^{m}\}_{m}$ and $\{t^{\beta-\beta_{0}}\theta^{m}\}_{m}$ are Cauchy sequences in $C([0,T_{2}];X^{\alpha})$ for $\alpha=\alpha_{1}, \alpha_{2}$ and in $C([0,T_{2}];Y^{\beta})$ for $\beta=\beta_{1}, \beta_{2}$ respectively.
\end{lemma}
\begin{proof}
It is obvious from $(2.32)_{2}$ that $1+\beta_{0}-\alpha_{0}-\beta_{1}=0$.
We must consider, instead of (3.9), the following inductive inequality:
\begin{equation}
\begin{split}
&K^{m+1}_{1,\alpha}(t)\leq K^{0}_{1,\alpha}(t)+CK^{m}_{1,\alpha_{1}}(t)^{2}+CK^{m}_{2,\beta_{1}}(t), \\
&K^{m+1}_{2,\beta}(t)\leq K^{0}_{2,\beta}(t)+C(K^{m}_{1,\alpha_{2}}(t)K^{m}_{2,\beta_{2}}(t)+K^{m}_{1,\alpha_{2}}(t)^{2})
\end{split}
\end{equation}
for any $\alpha=\alpha_{1}, \alpha_{2}$, $\beta=\beta_{1}, \beta_{2}$, $0<t\leq T$, $m \in \mathbb{Z}$, $m\geq 0$ which is derived from (3.7), (3.8).
It can be easily seen from (3.14) that
\begin{equation}
K^{m+2}(t)\leq C(K^{0}(t)+K^{m+1}(t)^{2}+K^{m}(t)^{2})
\end{equation}
for any $0<t\leq T$.
Since $K^{0}(0)=0$, $K^{m}\leq K^{m+1}\leq K^{m+2}$ on $[0,T]$, an elementary calculation shows that there exists a positive constant $\tau_{2}\leq T$ such that $K^{m}\leq CK^{0}$ on $[0,\tau_{2}]$.
It remains to give $X^{\alpha}$-estimates for $u^{m+1}-u^{m}$ and $Y^{\beta}$-estimates for $\theta^{m+1}-\theta^{m}$, but we can carry out the same proof as in Lemma 3.2.
\end{proof}
Set $T_{*}=\min\{T_{1}, T_{2}\}$.
Then it follows from Lemmas 3.2 and 3.3 that there exists a pair of two functions $(u,\theta)$ satisfying
\begin{equation*}
u \in C((0,T_{*}];X^{\alpha_{0}}),
\end{equation*}
\begin{equation*}
\theta \in C((0,T_{*}];Y^{\beta_{0}})
\end{equation*}
such that
\begin{equation}
\begin{cases}
t^{\alpha-\alpha_{0}}u^{m}\rightarrow t^{\alpha-\alpha_{0}}u & \mathrm{in} \ C([0,T_{*}];X^{\alpha}) \ \mathrm{as} \ m\rightarrow \infty, \\
t^{\beta-\beta_{0}}\theta^{m}\rightarrow t^{\beta-\beta_{0}}\theta & \mathrm{in} \ C([0,T_{*}];Y^{\beta}) \ \mathrm{as} \ m\rightarrow \infty
\end{cases}
\end{equation}
for $\alpha=\alpha_{1}, \alpha_{2}$, $\beta=\beta_{1}, \beta_{2}$.
By applying the dominated convergence theorem to (3.2), we can conclude that $(u,\theta)$ satisfies (II) in $(0,T_{*}]$.

\subsection{$X^{\alpha}\times Y^{\beta}$-estimates for local mild solutions}
We will deal with basic properties of local mild solutions of (1.1), (1.2).
It is sufficient for (2.11)--(2.14) that we prove the following lemma:
\begin{lemma}
Let $(u,\theta)$ be a mild solution of $(1.1)$, $(1.2)$ in $(0,T_{*}]$ given by $(3.16)$.
Then
\begin{equation}
t^{\alpha-\alpha_{0}}\|u(t)-e^{-tA}u_{0}\|_{X^{\alpha}}\leq CK^{0}(t),
\end{equation}
\begin{equation}
t^{\beta-\beta_{0}}\|\theta(t)-e^{-tB}\theta_{0}\|_{Y^{\beta}}\leq CK^{0}(t)
\end{equation}
for any $\alpha_{0}\leq\alpha<1$, $\beta_{0}\leq\beta<1$, $0<t\leq T_{*}$, where $C$ is a positive constant independent of $u$, $\theta$ and $t$.
\end{lemma}
\begin{proof}
It can be easily seen that (3.10), (3.11) with $(u,\theta)$ instead of $(u^{m},\theta^{m})$ hold for any $\alpha_{0}\leq\alpha<1-\delta_{1}$, $\beta_{0}\leq\beta<1-\delta_{2}$.
By applying (3.10), (3.11) to (2.34), (2.35), we have the following inequalities:
\begin{equation*}
t^{\alpha-\alpha_{0}}\|u(t)-e^{-tA}u_{0}\|_{X^{\alpha}}\leq CK^{0}(t)(K^{0}(t)+1),
\end{equation*}
\begin{equation*}
t^{\beta-\beta_{0}}\|\theta(t)-e^{-tB}\theta_{0}\|_{Y^{\beta}}\leq CK^{0}(t)^{2}
\end{equation*}
for any $\alpha_{0}\leq\alpha<1-\delta_{1}$, $\beta_{0}\leq\beta<1-\delta_{2}$, $0<t\leq T_{*}$.
Furthermore, the choice of $\delta_{1}$ and $\delta_{2}$ allows us to assume that $\alpha_{0}\leq \alpha<1$, $\beta_{0}\leq \beta<1$.
These inequalities clearly lead to (3.17), (3.18).
\end{proof}
It follows from (3.17) with $\alpha=\alpha_{0}$, (3.18) with $\beta=\beta_{0}$ that
\begin{equation}
\|u(t)-u_{0}\|_{X^{\alpha_{0}}}\leq \|(e^{-tA}-I)u_{0}\|_{X^{\alpha_{0}}}+CK^{0}(t),
\end{equation}
\begin{equation}
\|\theta(t)-\theta_{0}\|_{Y^{\beta_{0}}}\leq \|(e^{-tB}-I)\theta_{0}\|_{Y^{\beta_{0}}}+CK^{0}(t)
\end{equation}
for any $0<t\leq T_{*}$.
By taking $t$ as zero, (3.19), (3.20) imply that $u(0)=u_{0}$, $\theta(0)=\theta_{0}$, consequently, $(u,\theta)$ is a mild solution of (1.1), (1.2) on $[0,T_{*}]$.
It is obvious from (3.17), (3.18) that
\begin{equation}
t^{\alpha-\alpha_{0}}\|u(t)\|_{X^{\alpha}}\leq t^{\alpha-\alpha_{0}}\|e^{-tA}u_{0}\|_{X^{\alpha}}+CK^{0}(t),
\end{equation}
\begin{equation}
t^{\beta-\beta_{0}}\|\theta(t)\|_{Y^{\beta}}\leq t^{\beta-\beta_{0}}\|e^{-tB}\theta_{0}\|_{Y^{\beta}}+CK^{0}(t)
\end{equation}
for any $\alpha_{0}\leq\alpha<1$, $\beta_{0}\leq\beta<1$, $0<t\leq T_{*}$.
(2.1), (2.2), (3.21), (3.22) clearly yield (2.11), (2.12).
Moreover, it can be easily seen from (2.5), (2.6), (3.21), (3.22) that (2.13), (2.14) hold for any $\alpha_{0}<\alpha<1$, $\beta_{0}<\beta<1$.

\subsection{Uniqueness of mild solutions}
We proceed to the uniqueness of mild solutions of (1.1), (1.2) on $[0,T]$.
Throughout this subsection, it is required that $C$ is a positive constant independent of $t$, consequently, $C$ may depend on $u$ and $\theta$.
For any $\alpha_{0}\leq \alpha<1$, $\beta_{0}\leq \beta<1$, $0<\tau\leq t\leq T$, let us introduce the following notation:
\begin{equation*}
\|u(t)\|_{(\alpha)}=t^{\alpha-\alpha_{0}}\|u(t)\|_{X^{\alpha}},
\end{equation*}
\begin{equation*}
\|\theta(t)\|_{(\beta)}=t^{\beta-\beta_{0}}\|\theta(t)\|_{Y^{\beta}},
\end{equation*}
\begin{equation*}
\|(u,\theta)(t)\|_{(\alpha,\beta)}=\|u(t)\|_{(\alpha)}+\|\theta(t)\|_{(\beta)},
\end{equation*}
\begin{equation*}
\|u\|_{(\alpha;t)}=\sup_{0<s\leq t}s^{\alpha-\alpha_{0}}\|u(s)\|_{X^{\alpha}},
\end{equation*}
\begin{equation*}
\|\theta\|_{(\beta;t)}=\sup_{0<s\leq t}s^{\beta-\beta_{0}}\|\theta(s)\|_{Y^{\beta}},
\end{equation*}
\begin{equation*}
\|(u,\theta)\|_{(\alpha,\beta;t)}=\|u\|_{(\alpha;t)}+\|\theta\|_{(\beta;t)},
\end{equation*}
\begin{equation*}
\|u\|_{(\alpha;\tau,t)}=\max_{\tau\leq s\leq t}\|u(s)\|_{X^{\alpha}},
\end{equation*}
\begin{equation*}
\|\theta\|_{(\beta;\tau,t)}=\max_{\tau\leq s\leq t}\|\theta(s)\|_{Y^{\beta}},
\end{equation*}
\begin{equation*}
\|(u,\theta)\|_{(\alpha,\beta;\tau,t)}=\|u\|_{(\alpha;\tau,t)}+\|\theta\|_{(\beta;\tau,t)}.
\end{equation*}
It is clear that the uniqueness is derived from the continuous dependence with respect to initial data.
We prove the following lemma:
\begin{lemma}
Let $(u,\theta)$ and $(\bar{u},\bar{\theta})$ be two mild solutions of $(1.1)$, $(1.2)$ on $[0,T]$ with initial data $(u_{0},\theta_{0})$ and $(\bar{u}_{0},\bar{\theta}_{0})$ respectively which satisfy the following conditions:
\begin{itemize}
\item[\rm{(i)}]For any $\alpha_{0}\leq\alpha<1$, $\beta_{0}\leq\beta<1$,
\begin{equation*}
t^{\alpha-\alpha_{0}}u, \ t^{\alpha-\alpha_{0}}\bar{u} \in C([0,T];X^{\alpha}),
\end{equation*}
\begin{equation*}
t^{\beta-\beta_{0}}\theta, \ t^{\beta-\beta_{0}}\bar{\theta} \in C([0,T];Y^{\beta}).
\end{equation*}
\item[\rm{(ii)}]For any $\alpha_{0}<\alpha<1$, $\beta_{0}<\beta<1$,
\begin{equation*}
\|u(t)\|_{X^{\alpha}}=o(t^{\alpha_{0}-\alpha}), \ \|\bar{u}(t)\|_{X^{\alpha}}=o(t^{\alpha_{0}-\alpha}) \ \mathrm{as} \ t\rightarrow +0,
\end{equation*}
\begin{equation*}
\|\theta(t)\|_{Y^{\beta}}=o(t^{\beta_{0}-\beta}), \ \|\bar{\theta}(t)\|_{Y^{\beta}}=o(t^{\beta_{0}-\beta}) \ \mathrm{as} \ t\rightarrow +0.
\end{equation*}
\end{itemize}
Then
\begin{equation}
\|(u-\bar{u},\theta-\bar{\theta})(t)\|_{(\alpha,\beta)}\leq C\|(u_{0}-\bar{u}_{0},\theta_{0}-\bar{\theta}_{0})\|_{(\alpha_{0},\beta_{0})}
\end{equation}
for any $\alpha_{0}\leq\alpha<1$, $\beta_{0}\leq\beta<1$, $0<t\leq T$, where $C$ is a positive constant independent of $t$.
\end{lemma}
\begin{proof}
$D_{0}$, $D$ and $M$ are defined as
\begin{equation*}
D_{0}=\|(u_{0}-\bar{u}_{0},\theta_{0}-\bar{\theta}_{0})\|_{(\alpha_{0},\beta_{0})},
\end{equation*}
\begin{equation*}
D(t)=\max\{\|(u-\bar{u},\theta-\bar{\theta})\|_{(\alpha_{1},\beta_{1};t)}, \|(u-\bar{u},\theta-\bar{\theta})\|_{(\alpha_{2},\beta_{2};t)}\},
\end{equation*}
\begin{equation*}
M(t)=\max\{\|u\|_{(\alpha_{1};t)}, \|u\|_{(\alpha_{2};t)}, \|\bar{u}\|_{(\alpha_{1};t)}, \|\bar{u}\|_{(\alpha_{2};t)}, \|\theta\|_{(\beta_{2};t)}\}.
\end{equation*}
By applying (2.36), (2.37) to $(u,\theta)$ and $(\bar{u},\bar{\theta})$, we have the following inequalities:
\begin{equation}
\begin{split}
\|(u-\bar{u})(t)\|_{(\alpha)}\leq& C\|u_{0}-\bar{u}_{0}\|_{(\alpha_{0})}+CM(t)\|u-\bar{u}\|_{(\alpha_{1};t)} \\
&+Ct^{1+\beta_{0}-\alpha_{0}-\beta_{1}}\|\theta-\bar{\theta}\|_{(\beta_{1};t)},
\end{split}
\end{equation}
\begin{equation}
\|(\theta-\bar{\theta})(t)\|_{(\beta)}\leq C\|\theta_{0}-\bar{\theta}_{0}\|_{(\beta_{0})}+CM(t)\|(u-\bar{u},\theta-\bar{\theta})\|_{(\alpha_{2},\beta_{2};t)}
\end{equation}
for any $\alpha_{0}\leq\alpha<1-\delta_{1}$, $\beta_{0}\leq\beta<1-\delta_{2}$, $0<t\leq T$.
Moreover, it can be easily seen from (3.24) with $\alpha=\alpha_{1}, \alpha_{2}$, (3.25) with $\beta=\beta_{1}, \beta_{2}$ that
\begin{equation}
D(t)\leq CD_{0}+CN(t)D(t)
\end{equation}
for any $0<t\leq T$, where
\begin{equation*}
N(t):=\begin{cases}
M(t)+t^{1+\beta_{0}-\alpha_{0}-\beta_{1}} & \mathrm{if} \ (2.33)_{1}, \\
M(t) & \mathrm{if} \ (2.33)_{2}.
\end{cases}
\end{equation*}
Since $(u,\theta)$ and $(\bar{u},\bar{\theta})$ satisfy (i), (ii), $N$ is a monotone increasing continuous function on $[0,T]$ satisfying $N(0)=0$.
Then we can take a positive constant $\tau_{0}\leq T$ satisfying $CN(\tau_{0})<1$, consequently, $D(\tau_{0})\leq CD_{0}$.
It remains to prove (3.23) for any $\tau_{0}\leq t\leq T$.
For any $\tau_{0}\leq \tau\leq T$, $D(\tau,\cdot)$ and $M(\tau,\cdot)$ are defined as
\begin{equation*}
D(\tau,t)=\max\{\|(u-\bar{u},\theta-\bar{\theta})\|_{(\alpha_{1},\beta_{1};\tau,t)}, \|(u-\bar{u},\theta-\bar{\theta})\|_{(\alpha_{2},\beta_{2};\tau,t)}\},
\end{equation*}
\begin{equation*}
M(\tau,t)=\max\{\|u\|_{(\alpha_{1};\tau,t)}, \|u\|_{(\alpha_{2};\tau,t)}, \|\bar{u}\|_{(\alpha_{1};\tau,t)}, \|\bar{u}\|_{(\alpha_{2};\tau,t)}, \|\theta\|_{(\beta_{2};\tau,t)}\}.
\end{equation*}
It is necessary to remark that $(u,\theta)$ and $(\bar{u},\bar{\theta})$ satisfy
\begin{equation}
\begin{cases}
u(t)=e^{-(t-\tau)A}u(\tau)+\displaystyle\int^{t}_{\tau}e^{-(t-s)A}F(u,\theta)(s)ds, \\
\theta(t)=e^{-(t-\tau)B}\theta(\tau)+\displaystyle\int^{t}_{\tau}e^{-(t-s)B}G(u,\theta)(s)ds
\end{cases}
\end{equation}
for any $\tau\leq t\leq T$.
We subtract (3.27) with $(\bar{u},\bar{\theta})$ from (3.27) with $(u,\theta)$, and obtain
\begin{equation*}
\begin{split}
\|(u-\bar{u})(t)\|_{X^{\alpha}}\leq& \|e^{-(t-\tau)A}(u-\bar{u})(\tau)\|_{X^{\alpha}} \\
&+\int^{t}_{\tau}\|e^{-(t-s)A}(F(u,\theta)-F(\bar{u},\bar{\theta}))(s)\|_{X^{\alpha}}ds,
\end{split}
\end{equation*}
\begin{equation*}
\begin{split}
\|(\theta-\bar{\theta})(t)\|_{Y^{\beta}}\leq& \|e^{-(t-\tau)B}(\theta-\bar{\theta})(\tau)\|_{Y^{\beta}} \\
&+\int^{t}_{\tau}\|e^{-(t-s)B}(G(u,\theta)-G(\bar{u},\bar{\theta}))(s)\|_{Y^{\beta}}ds
\end{split}
\end{equation*}
for any $\alpha_{0}\leq\alpha<1-\delta_{1}$, $\beta_{0}\leq\beta<1-\delta_{2}$, $\tau\leq t\leq T$.
It is obvious that
\begin{equation*}
\int^{t}_{\tau}\|e^{-(t-s)A}(F(u,\theta)-F(\bar{u},\bar{\theta}))(s)\|_{X^{\alpha}}ds,
\end{equation*}
\begin{equation*}
\int^{t}_{\tau}\|e^{-(t-s)B}(G(u,\theta)-G(\bar{u},\bar{\theta}))(s)\|_{Y^{\beta}}ds
\end{equation*}
are estimated like (2.36), (2.37), consequently, we have the following inequalities:
\begin{equation}
\begin{split}
\|(u-\bar{u})(t)\|_{X^{\alpha}}\leq& C\tau^{\alpha_{0}-\alpha}_{0}D_{0}+C(t-\tau)^{1-(\alpha+\delta_{1})}M(\tau_{0},T)\|u-\bar{u}\|_{(\alpha_{1};\tau,t)} \\
&+C(t-\tau)^{1-\alpha}\|\theta-\bar{\theta}\|_{(\beta_{1};\tau,t)},
\end{split}
\end{equation}
\begin{equation}
\begin{split}
\|(\theta-\bar{\theta})(t)\|_{Y^{\beta}}\leq& C\tau^{\beta_{0}-\beta}_{0}D_{0} \\
&+C\{(t-\tau)^{1-(\beta+\delta_{2})}+(t-\tau)^{1-\beta}\}M(\tau_{0},T) \\
&\times \|(u-\bar{u},\theta-\bar{\theta})\|_{(\alpha_{2},\beta_{2};\tau,t)}
\end{split}
\end{equation}
for any $\tau\leq t\leq T$.
Similarly to (3.26), it follows from (3.28) with $\alpha=\alpha_{1}, \alpha_{2}$, (3.29) with $\beta=\beta_{1}, \beta_{2}$ that
\begin{equation}
D(\tau,t)\leq C(\tau^{\alpha_{0}-\alpha}_{0}+\tau^{\beta_{0}-\beta}_{0})D_{0}+CN(\tau,t)D(\tau,t)
\end{equation}
for any $\tau\leq t\leq T$, where
\begin{equation*}
N(\tau,t):=\{(t-\tau)^{1-(\alpha+\delta_{1})}+(t-\tau)^{1-(\beta+\delta_{2})}+(t-\tau)^{1-\beta}\}M(\tau_{0},T)+(t-\tau)^{1-\alpha}.
\end{equation*}
It is clear that there exists a positive constant $\tau_{1}\leq T-\tau$ independent of $\tau$ such that $CN(\tau,\tau+\tau_{1})<1$, consequently, $D(\tau,\tau+\tau_{1})\leq CD_{0}$.
We repeat to carry out the same proof as above, and obtain $D(\tau_{0},T)\leq CD_{0}$.
\end{proof}

\subsection{Existence of global mild solutions}
The main purpose of this subsection is to extend a mild solution of (1.1), (1.2) locally in time to the one globally in time.
By virtue of Theorem 2.1, it is essential for Theorem 2.2 that we obtain global $X^{\alpha}$-estimates (2.15) for $u$ and global $Y^{\beta}$-estimates (2.16) for $\theta$.
For any $0<\lambda<\Lambda_{1}$, $\lambda<\lambda_{1}<\Lambda_{1}$, $\lambda<\lambda_{2}<\min\{2\lambda, \lambda_{1}\}$, let us introduce monotone increasing continuous functions on $[0,\infty)$ defined as
\begin{equation*}
E_{1,\alpha}(t)=\sup_{0<s\leq t}s^{\alpha-\alpha_{0}}e^{\lambda s}\|u(s)\|_{X^{\alpha}},
\end{equation*}
\begin{equation*}
E_{2,\beta}(t)=\sup_{0<s\leq t}s^{\beta-\beta_{0}}e^{\lambda_{2}s}\|\theta(s)\|_{Y^{\beta}}.
\end{equation*}
It is clear that (2.15), (2.16) are established by proving the following lemma:
\begin{lemma}
There exists a positive constant $\varepsilon$ depending only on $n$, $\Omega$, $p$, $q$, $\alpha_{0}$, $\beta_{0}$, $L_{f}$ and $\lambda$ such that
\begin{equation}
E_{1,\alpha}(t)\leq C(\|u_{0}\|_{X^{\alpha_{0}}}+\|\theta_{0}\|_{Y^{\beta_{0}}}),
\end{equation}
\begin{equation}
E_{2,\beta}(t)\leq C(\|u_{0}\|_{X^{\alpha_{0}}}+\|\theta_{0}\|_{Y^{\beta_{0}}})
\end{equation}
for any $\alpha_{0}\leq\alpha<1$, $\beta_{0}\leq\beta<1$, $t>0$, where $C$ is a positive constant independent of $u$, $\theta$ and $t$ provided that
\begin{equation*}
\|u_{0}\|_{X^{\alpha_{0}}}+\|\theta_{0}\|_{Y^{\beta_{0}}}\leq\varepsilon.
\end{equation*}
\end{lemma}
\begin{proof}
It follows from $\mathrm{(II)}_{1}$, (2.34) that 
\begin{equation*}
\begin{split}
t^{\alpha-\alpha_{0}}e^{\lambda t}\|u(t)\|_{X^{\alpha}}\leq& C_{A,\alpha-\alpha_{0},\lambda_{1}}e^{-(\lambda_{1}-\lambda)t}\|u_{0}\|_{X^{\alpha_{0}}} \\
&+C_{A,\alpha+\delta_{1},\lambda_{1}}C_{1}t^{\alpha-\alpha_{0}}e^{-(\lambda_{1}-\lambda)t}E_{1,\alpha_{1}}(t)^{2} \\
&\times \int^{t}_{0}(t-s)^{-(\alpha+\delta_{1})}s^{-2(\alpha_{1}-\alpha_{0})}e^{-(2\lambda-\lambda_{1})s}ds \\
&+C_{A,\alpha,\lambda_{1}}C_{4}L_{f}t^{\alpha-\alpha_{0}}e^{-(\lambda_{1}-\lambda)t}E_{2,\beta_{1}}(t) \\
&\times \int^{t}_{0}(t-s)^{-\alpha}s^{-(\beta_{1}-\beta_{0})}e^{-(\lambda_{2}-\lambda_{1})s}ds \\
\leq& C_{A,\alpha-\alpha_{0},\lambda_{1}}\|u_{0}\|_{X^{\alpha_{0}}}+Ct^{1+\alpha_{0}-2\alpha_{1}-\delta_{1}}e^{-\lambda t}E_{1,\alpha_{1}}(t)^{2} \\
&+CL_{f}t^{1+\beta_{0}-\alpha_{0}-\beta_{1}}e^{-(\lambda_{2}-\lambda)t}E_{2,\beta_{1}}(t),
\end{split}
\end{equation*}
\begin{equation}
E_{1,\alpha}(t)\leq C(\|u_{0}\|_{X^{\alpha_{0}}}+E_{1,\alpha_{1}}(t)^{2}+L_{f}E_{2,\beta_{1}}(t))
\end{equation}
for any $\alpha_{0}\leq \alpha<1-\delta_{1}$, $t>0$.
Similarly to (3.33), we can utilize $\mathrm{(II)}_{2}$, (2.35) to obtain the following inequality:
\begin{equation*}
\begin{split}
t^{\beta-\beta_{0}}e^{\lambda_{2}t}\|\theta(t)\|_{Y^{\beta}}\leq& C_{B,\beta-\beta_{0},\lambda_{1}}e^{-(\lambda_{1}-\lambda_{2})t}\|\theta_{0}\|_{Y^{\beta_{0}}} \\
&+C_{B,\beta+\delta_{2},\lambda_{1}}C_{2}t^{\beta-\beta_{0}}e^{-(\lambda_{1}-\lambda_{2})t}E_{1,\alpha_{2}}(t)E_{2,\beta_{2}}(t) \\
&\times \int^{t}_{0}(t-s)^{-(\beta+\delta_{2})}s^{-(\alpha_{2}-\alpha_{0})-(\beta_{2}-\beta_{0})}e^{-(\lambda+\lambda_{2}-\lambda_{1})s}ds \\
&+C_{B,\beta,\lambda_{1}}C_{3}t^{\beta-\beta_{0}}e^{-(\lambda_{1}-\lambda_{2})t}E_{1,\alpha_{2}}(t)^{2} \\
&\times \int^{t}_{0}(t-s)^{-\beta}s^{-2(\alpha_{2}-\alpha_{0})}e^{-(2\lambda-\lambda_{1})s}ds \\
\leq& C_{B,\beta-\beta_{0},\lambda_{1}}\|\theta_{0}\|_{Y^{\beta_{0}}}+Ct^{1+\alpha_{0}-\alpha_{2}-\beta_{2}-\delta_{2}}e^{-\lambda t}E_{1,\alpha_{2}}(t)E_{2,\beta_{2}}(t) \\
&+Ct^{1+2\alpha_{0}-\beta_{0}-2\alpha_{2}}e^{-(2\lambda-\lambda_{2})t}E_{1,\alpha_{2}}(t)^{2},
\end{split}
\end{equation*}
\begin{equation}
E_{2,\beta}(t)\leq C(\|\theta_{0}\|_{Y^{\beta_{0}}}+E_{1,\alpha_{2}}(t)E_{2,\beta_{2}}(t)+E_{1,\alpha_{2}}(t)^{2})
\end{equation}
for any $\beta_{0}\leq \beta<1-\delta_{2}$, $t>0$.
Set $E(t)=\max\{E_{1,\alpha}(t), E_{2,\beta}(t) \ ; \ \alpha=\alpha_{1}, \alpha_{2}, \beta=\beta_{1}, \beta_{2}\}$.
Then (3.33), (3.34) yield the following inequality:
\begin{equation}
E(t)\leq C\{(\|u_{0}\|_{X^{\alpha_{0}}}+\|\theta_{0}\|_{Y^{\beta_{0}}})+E(t)^{2}\}
\end{equation}
for any $t>0$.
An elementary calculation shows that
\begin{equation}
E(t)\leq C(\|u_{0}\|_{X^{\alpha_{0}}}+\|\theta_{0}\|_{Y^{\beta_{0}}})
\end{equation}
for any $t>0$ provided that $\|u_{0}\|_{X^{\alpha_{0}}}$ and $\|\theta_{0}\|_{Y^{\beta_{0}}}$ are sufficiently small.
Therefore, it is clear from (3.36) that (3.31), (3.32) are established by (3.33), (3.34).
\end{proof}

\section{Proof of Theorems 2.3 and 2.4}
We will prove Theorems 2.3 and 2.4 in this section.
Since the proof of Theorem 2.4 is essentially the same as the proof of Theorem 2.3, we have only to prove Theorem 2.3.
Moreover, in proving Theorem 2.3, we restrict ourselves to the case where $\delta_{1}=0$, $\delta_{2}=0$.
Even if $\delta_{1}>0$ or $\delta_{2}>0$, it is sufficient for Theorem 2.3 that we slightly modify the argument in this section.

\subsection{$X^{\alpha}\times Y^{\beta}$-estimates for integrals}
Theorems 2.3 and 2.4 are established by the following lemmas:
\begin{lemma}[9, Lemma 3.4]
Let
\begin{equation}
\mathcal{F}(t)=\int^{t}_{0}e^{-(t-s)A}F(s)ds
\end{equation}
with $F \in C((0,T];L^{p}_{\sigma}(\Omega))$ satisfying
\begin{equation}
\|F(t)\|_{p}\leq C_{F}t^{-a}
\end{equation}
for any $0<t<t+h\leq T$, where $C_{F}$ is a positive constant, $0\leq a<1$.
Then
\begin{itemize}
\item[\rm{(i)}]For any $0\leq \alpha<1$, $0<\tilde{\alpha}<1-\alpha$,
\begin{equation*}
\mathcal{F} \in C^{0,\tilde{\alpha}}((0,T];X^{\alpha}).
\end{equation*}
\item[\rm{(ii)}]For any $0<t<t+h\leq T$,
\begin{equation}
\|\mathcal{F}(t+h)-\mathcal{F}(t)\|_{X^{\alpha}}\leq L_{\mathcal{F}}C_{F}(h^{1-\alpha}t^{-a}+h^{\tilde{\alpha}}t^{1-\alpha-\tilde{\alpha}-a}),
\end{equation}
where $L_{\mathcal{F}}=L_{\mathcal{F}}(\alpha,\tilde{\alpha})$ is a positive constant.
\end{itemize}
\end{lemma}
\begin{lemma}[3, Lemmas 2.13 and 2.14, 9, Lemma 3.5]
Let $\mathcal{F}$ be a integral given by $(4.1)$ with $F \in C((0,T];L^{p}_{\sigma}(\Omega))$ satisfying $(4.2)$ and
\begin{equation}
\|F(t+h)-F(t)\|_{p}\leq L_{F}h^{b}t^{-c}
\end{equation}
for any $0<t<t+h\leq T$, where $L_{F}$ is positive constant, $0<b\leq 1$, $c>0$.
Then
\begin{itemize}
\item[\rm{(i)}]For any $0<\hat{\alpha}<b$, $0\leq \alpha<b$, $0<\tilde{\alpha}<b-\alpha$,
\begin{equation*}
\mathcal{F} \in C^{0,\hat{\alpha}}((0,T];X^{1}), \ d_{t}\mathcal{F} \in C^{0,\tilde{\alpha}}((0,T];X^{\alpha}).
\end{equation*}
\item[\rm{(ii)}]For any $0<t\leq T$,
\begin{equation*}
d_{t}\mathcal{F}(t)+A\mathcal{F}(t)=F(t).
\end{equation*}
\item[\rm{(iii)}]For any $0<t\leq T$,
\begin{equation}
\|\mathcal{F}(t)\|_{X^{1}}\leq C_{1,\mathcal{F}}(C_{F}t^{-a}+L_{F}t^{b-c}),
\end{equation}
where $C_{1,\mathcal{F}}=C_{1,\mathcal{F}}(b,c)$ is a positive constant.
\item[\rm{(iv)}]For any $0\leq \alpha< b$, $0<t\leq T$,
\begin{equation}
\|d_{t}\mathcal{F}(t)\|_{X^{\alpha}}\leq C_{2,\mathcal{F}}(C_{F}t^{-(\alpha+a)}+L_{F}t^{b-(\alpha+c)}),
\end{equation}
where $C_{2,\mathcal{F}}=C_{2,\mathcal{F}}(\alpha,b,c)$ is a positive constant.
\end{itemize}
\end{lemma}
We remark that the regularity lemmas similar to Lemmas 4.1 and 4.2 are still valid for $B$, $G$ and $\mathcal{G}$ instead of $A$, $F$ and $\mathcal{F}$ respectively.

It is useful for the time derivative of strong solutions of (1.1), (1.2) to be stated the following generalized Gronwall lemma:
\begin{lemma}[9, Remark to Lemma 3.6]
Let $y$ be a nonnegative continuous and integrable function in $(0,T]$ satisfying
\begin{equation}
y(t)\leq \sum^{l}_{i=1}a_{i}t^{-\alpha_{i}}+\sum^{m}_{j=1}b_{j}\int^{t}_{0}(t-s)^{-\beta_{j}}y(s)ds
\end{equation}
for any $0<t\leq T$, where $a_{i}>0$, $b_{j}>0$, $0\leq \alpha_{i}<1$, $0\leq \beta_{j}<1$.
Then
\begin{equation}
y(t)\leq C\sum^{l}_{i=1}a_{i}t^{-\alpha_{i}}(1+B_{n_{\beta}+1}(t)e^{CB_{n_{\beta}+1}(t)})\sum^{n_{\beta}}_{k=0}B_{k}(t)
\end{equation}
for any $0<t\leq T$, where $C=C(\alpha_{1},\cdots,\alpha_{l},\beta_{1},\cdots,\beta_{m})$ is a positive constant, $n_{\beta}=[\beta/(1-\beta)]+1$, $\beta=\max\{\beta_{j} \ ; \ j=1,\cdots,m\}$,
\begin{equation*}
B_{k}(t)=\left(\sum^{m}_{j=1}b_{j}(t)\right)^{k}, \ b_{j}(t)=b_{j}t^{1-\beta_{j}}.
\end{equation*}
\end{lemma}

\subsection{Regularity of mild solutions}
We will show not only that a mild solution of (1.1), (1.2) can be a strong solution but also that (2.19), (2.20) are established.
It can be easily seen from Lemmas 2.5--2.8, (2.11), (2.12) that
\begin{equation}
\|F(u,\theta)(t)\|_{p}\leq C(t^{2(\alpha_{0}-\alpha_{1})}+t^{\beta_{0}-\beta_{1}})(\|u_{0}\|_{X^{\alpha_{0}}}+\|\theta_{0}\|_{Y^{\beta_{0}}}),
\end{equation}
\begin{equation}
\|G(u,\theta)(t)\|_{q}\leq C(t^{\alpha_{0}+\beta_{0}-\alpha_{2}-\beta_{2}}+t^{2(\alpha_{0}-\alpha_{2})})(\|u_{0}\|_{X^{\alpha_{0}}}+\|\theta_{0}\|_{Y^{\beta_{0}}})
\end{equation}
for any $0<t\leq T$.
Since
\begin{equation*}
u(t+h)-u(t)=(e^{-hA}-I)e^{-tA}u_{0}+\mathcal{F}(u,\theta)(t+h)-\mathcal{F}(u,\theta)(t),
\end{equation*}
\begin{equation*}
\theta(t+h)-\theta(t)=(e^{-hB}-I)e^{-tB}\theta_{0}+\mathcal{G}(u,\theta)(t+h)-\mathcal{G}(u,\theta)(t)
\end{equation*}
for any $0<t<t+h\leq T$, it follows from (4.3), (4.9), (4.10) that
\begin{equation}
\begin{split}
\|u(t+h)-u(t)\|_{X^{\alpha}}\leq& C(h^{b_{1}}t^{\alpha_{0}-\alpha-b_{1}}+h^{1-\alpha}t^{2(\alpha_{0}-\alpha_{1})}+h^{1-\alpha}t^{\beta_{0}-\beta_{1}}) \\
&\times (\|u_{0}\|_{X^{\alpha_{0}}}+\|\theta_{0}\|_{Y^{\beta_{0}}}),
\end{split}
\end{equation}
\begin{equation}
\begin{split}
\|\theta(t+h)-\theta(t)\|_{Y^{\beta}}\leq& C(h^{b_{2}}t^{\beta_{0}-\beta-b_{2}}+h^{1-\beta}t^{\alpha_{0}+\beta_{0}-\alpha_{2}-\beta_{2}}+h^{1-\beta}t^{2(\alpha_{0}-\alpha_{2})}) \\
&\times (\|u_{0}\|_{X^{\alpha_{0}}}+\|\theta_{0}\|_{Y^{\beta_{0}}})
\end{split}
\end{equation}
for any $0<b_{1}<1-\alpha$, $0<b_{2}<1-\beta$, $0<t<t+h\leq T$.
It is derived from (4.11) with $\alpha=\alpha_{1}, \alpha_{2}$, (4.12) with $\beta=\beta_{1}, \beta_{2}$ that
\begin{equation}
F(u,\theta) \in C^{0,\hat{\alpha}}((0,T];L^{p}_{\sigma}(\Omega)),
\end{equation}
\begin{equation}
G(u,\theta) \in C^{0,\hat{\beta}}((0,T];L^{q}(\Omega))
\end{equation}
for any $0<\hat{\alpha}<\min\{1-\alpha_{1}, 1-\beta_{1}\}$, $0<\hat{\beta}<\min\{1-\alpha_{2}, 1-\beta_{2}\}$.
Therefore, Lemma 4.2 (i), (ii) admit that $(u,\theta)$ is a strong solution of (1.1), (1.2).
By applying (4.5) to (II), we have the following inequalities:
\begin{equation}
t^{1-\alpha_{0}}\|u(t)\|_{X^{1}}\leq CM_{1}(t)(\|u_{0}\|_{X^{\alpha_{0}}}+\|\theta_{0}\|_{Y^{\beta_{0}}}),
\end{equation}
\begin{equation}
t^{1-\beta_{0}}\|\theta(t)\|_{Y^{1}}\leq CM_{2}(t)(\|u_{0}\|_{X^{\alpha_{0}}}+\|\theta_{0}\|_{Y^{\beta_{0}}})
\end{equation}
for any $0<t\leq T$, where
\begin{equation*}
\begin{split}
M_{1}(t):=&1+t^{1+\alpha_{0}-2\alpha_{1}}+t^{2(1+\alpha_{0}-2\alpha_{1})}+t^{2+\beta_{0}-2\alpha_{1}-\beta_{1}} \\
&+t^{1+\beta_{0}-\alpha_{0}-\beta_{1}}+t^{2+\beta_{0}-\beta_{1}-\alpha_{2}-\beta_{2}}+t^{2+\alpha_{0}-\beta_{1}-2\alpha_{2}},
\end{split}
\end{equation*}
\begin{equation*}
\begin{split}
M_{2}(t):=&1+t^{1+\alpha_{0}-\alpha_{2}-\beta_{2}}+t^{2+2\alpha_{0}-2\alpha_{1}-\alpha_{2}-\beta_{2}}+t^{2+\beta_{0}-\beta_{1}-\alpha_{2}-\beta_{2}} \\
&+t^{1+2\alpha_{0}-\beta_{0}-2\alpha_{2}}+t^{2+3\alpha_{0}-\beta_{0}-2\alpha_{1}-2\alpha_{2}}+t^{2+\alpha_{0}-\beta_{1}-2\alpha_{2}} \\
&+t^{2(1+\alpha_{0}-\alpha_{2}-\beta_{2})}+t^{2+3\alpha_{0}-\beta_{0}-3\alpha_{2}-\beta_{2}}.
\end{split}
\end{equation*}
It is clear from (2.32), (2.33) that (2.19), (2.20) are established by (4.15), (4.16).
\begin{remark}
Theorem $2.3$ $\mathrm{(i)}$ can be, more precisely, stated as follows:
\begin{equation*}
u \in C^{0,\hat{\alpha}}((0,T];X^{1}), \ d_{t}u \in C^{0,\tilde{\alpha}}((0,T];X^{\alpha}),
\end{equation*}
\begin{equation*}
\theta \in C^{0,\hat{\beta}}((0,T];Y^{1}), \ d_{t}\theta \in C^{0,\tilde{\beta}}((0,T];Y^{\beta})
\end{equation*}
for any $0<\hat{\alpha}<\min\{1-\alpha_{1}, 1-\beta_{1}\}$, $0<\hat{\beta}<\min\{1-\alpha_{2}, 1-\beta_{2}\}$, $0\leq\alpha<\min\{1-\alpha_{1}, 1-\beta_{1}\}$, $0\leq\beta<\min\{1-\alpha_{2}, 1-\beta_{2}\}$, $0<\tilde{\alpha}<\min\{1-\alpha_{1}, 1-\beta_{1}\}-\alpha$, $0<\tilde{\beta}<\min\{1-\alpha_{2}, 1-\beta_{2}\}-\beta$.
\end{remark}

\subsection{Regularity of the time derivative of strong solutions}
We will obtain the stronger regularity of strong solutions of (1.1), (1.2) under appropriate assumptions for $p$, $q$, $\alpha_{0}$ and $\beta_{0}$.
Let us remark that $(u,\theta)$ satisfies (3.27) for any $0<\tau<t<T$.
Then it can be easily seen from (3.27) that
\begin{equation}
\begin{split}
u(t+h)-u(t)=&(e^{-hA}-I)e^{-tA}u(\tau) \\
&+\int^{\tau+h}_{\tau}e^{-(t+h-s)A}F(u,\theta)(s)ds \\
&+\int^{t}_{\tau}e^{-(t-s)A}(F(u,\theta)(s+h)-F(u,\theta)(s))ds,
\end{split}
\end{equation}
\begin{equation}
\begin{split}
\theta(t+h)-\theta(t)=&(e^{-hB}-I)e^{-tB}\theta(\tau) \\
&+\int^{\tau+h}_{\tau}e^{-(t+h-s)B}G(u,\theta)(s)ds \\
&+\int^{t}_{\tau}e^{-(t-s)B}(G(u,\theta)(s+h)-G(u,\theta)(s))ds
\end{split}
\end{equation}
for any $\tau<t<t+h\leq T$.
It follows from (2.1), (2.3), (2.19), (4.9) that
\begin{equation*}
\|(e^{-hA}-I)e^{-tA}u(\tau)\|_{X^{\alpha}}\leq Ch(t-\tau)^{-\alpha}\tau^{\alpha_{0}-1}(\|u_{0}\|_{X^{\alpha_{0}}}+\|\theta_{0}\|_{Y^{\beta_{0}}}),
\end{equation*}
\begin{equation*}
\begin{split}
&\int^{t+\tau}_{\tau}\|e^{-(t+h-s)A}F(u,\theta)(s)\|_{X^{\alpha}}ds \\
&\leq C\int^{\tau+h}_{\tau}(t+h-s)^{-\alpha}(s^{2(\alpha_{0}-\alpha_{1})}+s^{\beta_{0}-\beta_{1}})ds(\|u_{0}\|_{X^{\alpha_{0}}}+\|\theta_{0}\|_{Y^{\beta_{0}}}) \\
&\leq Ch(t-\tau)^{-\alpha}(\tau^{2(\alpha_{0}-\alpha_{1})}+\tau^{\beta_{0}-\beta_{1}})(\|u_{0}\|_{X^{\alpha_{0}}}+\|\theta_{0}\|_{Y^{\beta_{0}}})
\end{split}
\end{equation*}
for any $\alpha_{0}\leq \alpha<1$, $\tau<t<t+h\leq T$.
Similarly to $u$ and $F(u,\theta)$, we can utilize (2.2), (2.4), (2.20), (4.10) to obtain that
\begin{equation*}
\|(e^{-hB}-I)e^{-tB}\theta(\tau)\|_{Y^{\beta}}\leq Ch(t-\tau)^{-\beta}\tau^{\beta_{0}-1}(\|u_{0}\|_{X^{\alpha_{0}}}+\|\theta_{0}\|_{Y^{\beta_{0}}}),
\end{equation*}
\begin{equation*}
\begin{split}
&\int^{t+\tau}_{\tau}\|e^{-(t+h-s)B}G(u,\theta)(s)\|_{Y^{\beta}}ds \\
&\leq C\int^{\tau+h}_{\tau}(t+h-s)^{-\beta}(s^{\alpha_{0}+\beta_{0}-\alpha_{2}-\beta_{2}}+s^{2(\alpha_{0}-\alpha_{2})})ds(\|u_{0}\|_{X^{\alpha_{0}}}+\|\theta_{0}\|_{Y^{\beta_{0}}}) \\
&\leq Ch(t-\tau)^{-\beta}(\tau^{\alpha_{0}+\beta_{0}-\alpha_{2}-\beta_{2}}+\tau^{2(\alpha_{0}-\alpha_{2})})(\|u_{0}\|_{X^{\alpha_{0}}}+\|\theta_{0}\|_{Y^{\beta_{0}}})
\end{split}
\end{equation*}
for any $\beta_{0}\leq \beta<1$, $\tau<t<t+h\leq T$.
Therefore, it follows from (4.17), (4.18) that
\begin{equation}
\begin{split}
\|u(t+h)-u(t)\|_{X^{\alpha}}\leq& C_{1}(\tau)h(t-\tau)^{-\alpha}\tau^{\alpha_{0}-1}(\|u_{0}\|_{X^{\alpha_{0}}}+\|\theta_{0}\|_{Y^{\beta_{0}}}) \\
&+C_{A,\alpha,\lambda}\int^{t}_{\tau}(t-s)^{-\alpha}\|F(u,\theta)(s+h)-F(u,\theta)(s)\|_{p}ds,
\end{split}
\end{equation}
\begin{equation}
\begin{split}
\|\theta(t+h)-\theta(t)\|_{Y^{\beta}}\leq& C_{2}(\tau)h(t-\tau)^{-\beta}\tau^{\beta_{0}-1}(\|u_{0}\|_{X^{\alpha_{0}}}+\|\theta_{0}\|_{Y^{\beta_{0}}}) \\
&+C_{B,\beta,\lambda}\int^{t}_{\tau}(t-s)^{-\beta}\|G(u,\theta)(s+h)-G(u,\theta)(s)\|_{q}ds
\end{split}
\end{equation}
for any $\tau<t<t+h\leq T$, where
\begin{equation*}
C_{1}(\tau):=C(1+\tau^{1+\alpha_{0}-2\alpha_{1}}+\tau^{1+\beta_{0}-\alpha_{0}-\beta_{1}}),
\end{equation*}
\begin{equation*}
C_{2}(\tau):=C(1+\tau^{1+\alpha_{0}-\alpha_{2}-\beta_{2}}+\tau^{1+2\alpha_{0}-\beta_{0}-2\alpha_{2}}).
\end{equation*}
Moreover, we can obtain $L^{p}_{\sigma}$-estimates for $F(u,\theta)(t+h)-F(u,\theta)(t)$ and $L^{q}$-estimates for $G(u,\theta)(t+h)-G(u,\theta)(t)$ with the aid of (4.19), (4.20), consequently,
\begin{equation}
\begin{split}
\|F&(u,\theta)(t+h)-F(u,\theta)(t)\|_{p} \\
\leq& Ch\{(t-\tau)^{-\alpha_{1}}\tau^{2\alpha_{0}-\alpha_{1}-1}+(t-\tau)^{-\beta_{1}}\tau^{\beta_{0}-1}\}(\|u_{0}\|_{X^{\alpha_{0}}}+\|\theta_{0}\|_{Y^{\beta_{0}}}) \\
&+C\tau^{\alpha_{0}-\alpha_{1}}\int^{t}_{\tau}(t-s)^{-\alpha_{1}}\|F(u,\theta)(s+h)-F(u,\theta)(s)\|_{p}ds \\
&+C\int^{t}_{\tau}(t-s)^{-\beta_{1}}\|G(u,\theta)(s+h)-G(u,\theta)(s)\|_{q}ds,
\end{split}
\end{equation}
\begin{equation}
\begin{split}
\|G&(u,\theta)(t+h)-G(u,\theta)(t)\|_{q} \\
\leq& Ch\{(t-\tau)^{-\alpha_{2}}(\tau^{2\alpha_{0}-\alpha_{2}-1}+\tau^{\alpha_{0}+\beta_{0}-\beta_{2}-1})+(t-\tau)^{-\beta_{2}}\tau^{\alpha_{0}+\beta_{0}-\alpha_{2}-1}\} \\
&\times (\|u_{0}\|_{X^{\alpha_{0}}}+\|\theta_{0}\|_{Y^{\beta_{0}}}) \\
&+C(\tau^{\alpha_{0}-\alpha_{2}}+\tau^{\beta_{0}-\beta_{2}})\int^{t}_{\tau}(t-s)^{-\alpha_{2}}\|F(u,\theta)(s+h)-F(u,\theta)(s)\|_{p}ds \\
&+C\tau^{\alpha_{0}-\alpha_{2}}\int^{t}_{\tau}(t-s)^{-\beta_{2}}\|G(u,\theta)(s+h)-G(u,\theta)(s)\|_{q}ds
\end{split}
\end{equation}
for any $\tau<t<t+h\leq T$.
Let $p$, $q$, $\alpha_{0}$ and $\beta_{0}$ satisfy (2.17), and set
\begin{equation*}
y(t)=\|F(u,\theta)(t+h)-F(u,\theta)(t)\|_{p}+\|G(u,\theta)(t+h)-G(u,\theta)(t)\|_{q}.
\end{equation*}
By applying Lemma 4.3 for $(\tau,T-h]$ instead of $(0,T]$ to (4.21), (4.22) and letting $\tau=t/2$, we have the following inequality:
\begin{equation}
y(t)\leq ChM(t)(\|u_{0}\|_{X^{\alpha_{0}}}+\|\theta_{0}\|_{Y^{\beta_{0}}})
\end{equation}
for any $0<t\leq T-h$, where
\begin{equation*}
M(t):=t^{2(\alpha_{0}-\alpha_{1})-1}+t^{\beta_{0}-\beta_{1}-1}+t^{\alpha_{0}+\beta_{0}-\alpha_{2}-\beta_{2}-1}+t^{2(\alpha_{0}-\alpha_{2})-1}.
\end{equation*}
It is clear from (4.23) that
\begin{equation*}
F(u,\theta) \in C^{0,1}((0,T];L^{p}_{\sigma}(\Omega)),
\end{equation*}
\begin{equation*}
G(u,\theta) \in C^{0,1}((0,T];L^{q}(\Omega)).
\end{equation*}
Therefore, Lemma 4.2 (i) yields Theorem 2.3 (ii).
Let $p$, $q$, $\alpha_{0}$ and $\beta_{0}$ satisfy (2.18) in addition to (2.17).
By applying (4.6) to (II), it follows from (4.9), (4.10), (4.23) that
\begin{equation}
t^{1+\alpha-\alpha_{0}}\|d_{t}u(t)\|_{X^{\alpha}}\leq CM_{1}(t)(\|u_{0}\|_{X^{\alpha_{0}}}+\|\theta_{0}\|_{Y^{\beta_{0}}}),
\end{equation}
\begin{equation}
t^{1+\beta-\beta_{0}}\|d_{t}\theta(t)\|_{Y^{\beta}}\leq CM_{2}(t)(\|u_{0}\|_{X^{\alpha_{0}}}+\|\theta_{0}\|_{Y^{\beta_{0}}})
\end{equation}
for $0<t\leq T$, where
\begin{equation*}
M_{1}(t):=1+t^{1+\alpha_{0}-2\alpha_{1}}+t^{1+\beta_{0}-\alpha_{0}-\beta_{1}}+t^{1+\beta_{0}-\alpha_{2}-\beta_{2}}+t^{1+\alpha_{0}-2\alpha_{2}},
\end{equation*}
\begin{equation*}
M_{2}(t):=1+t^{1+\alpha_{0}-\alpha_{2}-\beta_{2}}+t^{1+2\alpha_{0}-\beta_{0}-2\alpha_{2}}+t^{1+2\alpha_{0}-\beta_{0}-2\alpha_{1}}+t^{1-\beta_{1}}.
\end{equation*}
It is obvious from (2.32), (2.33) that (2.21), (2.22) are established by (4.24), (4.25).
\begin{remark}
Even if $p$, $q$, $\alpha_{0}$ and $\beta_{0}$ satisfy only $(2.9)$, $(2.10)$, it can be easily seen from $(4.13)$, $(4.14)$ that $(2.21)$, $(2.22)$ hold for any $0\leq\alpha<\min\{1-\alpha_{1}, 1-\beta_{1}\}$, $0\leq\beta<\min\{1-\alpha_{2}, 1-\beta_{2}\}$.
\end{remark}

\section{Proof of Corollaries 2.1 and 2.2}
We will prove Corollaries 2.1 and 2.2 in this section.
Since Corollary 2.2 is proved the same as in Corollary 2.1, it is essential for Corollaries 2.1 and 2.2 that we prove Corollary 2.1.

\subsection{$(W^{k,p})^{n}\times W^{k,q}$-estimates for nonlinear terms}
We will state and prove some lemmas for $(W^{k,p})^{n}\times W^{k,q}$-estimates.
It is assured by them that we establish $(W^{k,p})^{n}$-estimates for $F(u,\theta)$ and $W^{k,q}$-estimates for $G(u,\theta)$.
\begin{lemma}[7, Lemma 3.3]
\rm{(i)} Let $n<p<\infty$.
Then
\begin{equation}
\|P(u\cdot\nabla)v\|_{p}\leq C\|u\|_{1,p}\|v\|_{1,p}
\end{equation}
for any $u, v \in (W^{1,p}(\Omega))^{n}$, where $C=C(p)$ is a positive constant.

\rm{(ii)} Let $k \in \mathbb{Z}$, $k>n/p$.
Then
\begin{equation}
\|P(u\cdot\nabla)v\|_{k,p}\leq C\|u\|_{k,p}\|v\|_{k+1,p}
\end{equation}
for any $u \in (W^{k,p}(\Omega))^{n}$, $v \in (W^{k+1,p}(\Omega))^{n}$, where $C=C(k,p)$ is a positive constant.
\end{lemma}
\begin{lemma}
\rm{(i)} Let $n<p<\infty$, $1<q<\infty$.
Then
\begin{equation}
\|(u\cdot\nabla)\theta\|_{q}\leq C\|u\|_{1,p}\|\theta\|_{1,q}
\end{equation}
for any $u \in (W^{1,p}(\Omega))^{n}$, $\theta \in W^{1,q}(\Omega)$, where $C=C(p,q)$ is a positive constant.

\rm{(ii)} Let $n<p<\infty$, $n<q<\infty$, $q\leq p$, $k \in \mathbb{Z}$, $k>n/q$.
Then
\begin{equation}
\|(u\cdot\nabla)\theta\|_{k,q}\leq C\|u\|_{k,p}\|\theta\|_{k+1,q}
\end{equation}
for any  $u \in (W^{k,p}(\Omega))^{n}$, $\theta \in W^{k+1,q}(\Omega)$, where $C=C(k,p,q)$ is a positive constant.
\end{lemma}
\begin{proof}
(i) Let us notice that $W^{1,p}(\Omega)\hookrightarrow C(\overline{\Omega})$ from the Sobolev embedding theorem.
Then we obtain that
\begin{equation*}
\begin{split}
\|(u\cdot\nabla)\theta\|_{q}&\leq C\|u\|_{\infty}\|\theta\|_{1,q} \\
&\leq C\|u\|_{1,p}\|\theta\|_{1,q}.
\end{split}
\end{equation*}

(ii) It is known in \cite[Theorem 4.39]{Adams} that $W^{k,q}(\Omega)$ is a Banach algebra for any $k \in \mathbb{Z}$, $k>n/q$.
Therefore, the conclusion follows immediately from the above fact and $q\leq p$.
\end{proof}
\begin{lemma}
\rm{(i)} Let $n<p<\infty$, $1<q<\infty$, $2q\leq p$.
Then
\begin{equation}
\|\Phi(u,v)\|_{q}\leq C\|u\|_{1,p}\|v\|_{1,p}
\end{equation}
for any $u, v \in (W^{1,p}(\Omega))^{n}$, where $C=C(p,q)$ is a positive constant.

\rm{(ii)} Let $n<p<\infty$, $1<q<\infty$, $2q\leq p$, $k \in \mathbb{Z}$, $k>n/p$.
Then
\begin{equation}
\|\Phi(u,v)\|_{k,q}\leq C\|u\|_{k+1,p}\|v\|_{k+1,p}
\end{equation}
for any $u, v \in (W^{k+1,p}(\Omega))^{n}$, where $C=C(k,p,q)$ is a positive constant.
\end{lemma}
\begin{proof}
(i) After applying the Schwarz inequality to $\|\Phi(u,v)\|_{q}$, we can obtain (5.5) by $W^{1,p}(\Omega)\hookrightarrow W^{1,2q}(\Omega)$.

(ii) It can be easily seen from the Leibniz rule and the Schwarz inequality that $\|\Phi(u,v)\|_{k,q}\leq C\|u\|_{k+1,2q}\|v\|_{k+1,2q}$.
Therefore, $W^{k+1,p}(\Omega)\hookrightarrow W^{k+1,2q}(\Omega)$ implies (5.6).
\end{proof}
\begin{lemma}
\rm{(i)} Let $f \in C^{0,1}(\mathbb{R};\mathbb{R}^{n})$ with the Lipschitz constant $L_{f}$, $f(0)=0$, $n<p<\infty$, $n<q<\infty$, $q\leq p$.
Then
\begin{equation}
\|Pf(\theta)\|_{p}\leq CL_{f}\|\theta\|_{1,q}
\end{equation}
for any $\theta \in W^{1,q}(\Omega)$, where $C=C(p,q)$ is a positive constant.

\rm{(ii)} Let $f \in C^{0,1}(\mathbb{R};\mathbb{R}^{n})\cap C^{1}(\mathbb{R};\mathbb{R}^{n})$ with the Lipschitz constant $L_{f}$, $f(0)=0$, $n<p<\infty$, $n<q<\infty$, $q\leq p$.
Then
\begin{equation}
\|Pf(\theta)\|_{1,p}\leq C\|\theta\|_{2,q}
\end{equation}
for any $\theta \in W^{2,q}(\Omega)$, where $C=C(p,q)$ is a positive constant.
\end{lemma}
\begin{proof}
(i) Since $P$ is a bounded operator in $(L^{p}(\Omega))^{n}$ and $\|f(\theta)\|_{p}\leq L_{f}\|\theta\|_{p}$, it follows from $W^{1,q}(\Omega)\hookrightarrow L^{p}(\Omega)$ that we obtain (5.7).

(ii) It is known in \cite[Lemma 3.3]{Giga 3} that $P$ is a bounded operator not only in $(L^{p}(\Omega))^{n}$ but also in $(W^{1,p}(\Omega))^{n}$.
Since $f \in C^{0,1}(\mathbb{R};\mathbb{R}^{n})\cap C^{1}(\mathbb{R};\mathbb{R}^{n})$ implies $f' \in C_{b}(\mathbb{R};\mathbb{R}^{n})$, $\|f(\theta)\|_{1,p}\leq C\|\theta\|_{1,p}$.
Therefore, (5.8) follows immediately from $W^{2,q}(\Omega)\hookrightarrow W^{1,p}(\Omega)$.
\end{proof}

\subsection{Regularity of strong solutions}
It is sufficient for Corollary 2.1 that we obtain the following lemmas:
\begin{lemma}
Let $f \in C^{0,1}(\mathbb{R};\mathbb{R}^{n})$, $f(0)=0$, $p$ and $q$ satisfy $(2.27)$.
Then
\begin{equation*}
F(u,\theta) \in C^{0,\hat{\alpha}}((0,T];(W^{1,p}(\Omega))^{n}),
\end{equation*}
\begin{equation*}
G(u,\theta) \in C^{0,\hat{\beta}}((0,T];W^{1,q}(\Omega))
\end{equation*}
for any $0<\hat{\alpha}<1$, $0<\hat{\beta}<1$.
\end{lemma}
\begin{proof}
It follows immediately from Theorem 2.3 (ii), Lemmas 5.1--5.4 (ii) with $k=1$.
\end{proof}
\begin{lemma}
Let $f \in C^{0,1}(\mathbb{R};\mathbb{R}^{n})$, $f(0)=0$, $p$ and $q$ satisfy $(2.27)$.
Then
\begin{equation*}
u \in C^{0,\hat{\alpha}}((0,T];(W^{3,p}(\Omega))^{n}),
\end{equation*}
\begin{equation*}
\theta \in C^{0,\hat{\beta}}((0,T];W^{3,q}(\Omega))
\end{equation*}
for any $0<\hat{\alpha}<1/2$, $0<\hat{\beta}<1/2$.
\end{lemma}
\begin{proof}
It is clear from (2.7), (2.8) that $X^{1/2}\hookrightarrow (W^{1,p}(\Omega))^{n}$, $Y^{1/2}\hookrightarrow W^{1,q}(\Omega)$.
By applying Theorem 2.3 (ii) with $\alpha=1/2$, $\beta=1/2$ to $(d_{t}u,d_{t}\theta)$, the conclusion follows immediately from Lemma 5.5, $u=A^{-1}(F(u,\theta)-d_{t}u)$, $\theta=B^{-1}(G(u,\theta)-d_{t}\theta)$.
\end{proof}
\begin{lemma}
Let $f \in C^{0,1}(\mathbb{R};\mathbb{R}^{n})\cap C^{1}(\mathbb{R};\mathbb{R}^{n})$, $f(0)=0$, $p$ and $q$ satisfy $(2.27)$.
Then
\begin{equation*}
d_{t}F(u,\theta) \in C^{0,\hat{\alpha}}((0,T];L^{p}_{\sigma}(\Omega)),
\end{equation*}
\begin{equation*}
d_{t}G(u,\theta) \in C^{0,\hat{\beta}}((0,T];L^{q}(\Omega))
\end{equation*}
for any $0<\hat{\alpha}<\min\{1-\alpha_{1}, 1-\beta_{1}\}$, $0<\hat{\beta}<\min\{1-\alpha_{2}, 1-\beta_{2}\}$.
\end{lemma}
\begin{proof}
It follows immediately from Theorem 2.3 (ii) with $\alpha=\alpha_{1},\alpha_{2}$, $\beta=\beta_{1},\beta_{2}$, Lemmas 5.1--5.4 (i).
\end{proof}
\begin{lemma}
Let $f \in C^{0,1}(\mathbb{R};\mathbb{R}^{n})\cap C^{1}(\mathbb{R};\mathbb{R}^{n})$, $f(0)=0$, $p$ and $q$ satisfy $(2.27)$.
Then
\begin{equation*}
d_{t}u \in C^{0,\hat{\alpha}}((0,T];X^{1}), \ d^{2}_{t}u \in C^{0,\tilde{\alpha}}((0,T];X^{\alpha}),
\end{equation*}
\begin{equation*}
d_{t}\theta \in C^{0,\hat{\beta}}((0,T];Y^{1}), \ d^{2}_{t}\theta \in C^{0,\tilde{\beta}}((0,T];Y^{\beta})
\end{equation*}
for any $0<\hat{\alpha}<\min\{1-\alpha_{1}, 1-\beta_{1}\}$, $0<\hat{\beta}<\min\{1-\alpha_{2}, 1-\beta_{2}\}$, $0\leq\alpha<\min\{1-\alpha_{1}, 1-\beta_{1}\}$, $0\leq\beta<\min\{1-\alpha_{2}, 1-\beta_{2}\}$, $0<\tilde{\alpha}<\min\{1-\alpha_{1}, 1-\beta_{1}\}-\alpha$, $0<\tilde{\beta}<\min\{1-\alpha_{2}, 1-\beta_{2}\}-\beta$.
\end{lemma}
\begin{proof}
Lemma 5.7 admits that we differentiate (I) with respect to $t$ and obtain the following abstract integral equations:
\begin{equation*}
\begin{cases}
d_{t}u(t)=e^{-(t-\tau)A}d_{t}u(\tau)+\displaystyle\int^{t}_{\tau}e^{-(t-s)A}d_{t}F(u,\theta)(s)ds, \\
d_{t}\theta(t)=e^{-(t-\tau)B}d_{t}\theta(\tau)+\displaystyle\int^{t}_{\tau}e^{-(t-s)B}d_{t}G(u,\theta)(s)ds
\end{cases}
\end{equation*}
for any $0<\tau\leq t\leq T$.
Therefore, the conclusion follows immediately from Lemmas 4.2 (\cite[Lemma 2.14]{Fujita}) and 5.7.
\end{proof}
It follows from the Sobolev embedding theorem that $W^{k+1,p}(\Omega)\hookrightarrow C^{k,\alpha}(\overline{\Omega})$, $W^{k+1,q}(\Omega)$ $\hookrightarrow C^{k,\beta}(\overline{\Omega})$ for any $k \in \mathbb{Z}$, $k\geq 0$, $0<\alpha<1-n/p$, $0<\beta<1-n/q$.
Therefore, Lemmas 5.6 and 5.8 imply Corollary 2.1.


\begin{thebibliography}{99}

\bibitem[1]{Adams}R. A. Adams and J. J. F. Fournier, Sobolev Spaces (Second Edition), Academic Press, 2003.

\bibitem[2]{Boussinesq}J. Boussinesq, Th\'{e}orie Analytique de la Chaleur. II, Gauthier-Villars, 1903.

\bibitem[3]{Fujita}H. Fujita and T. Kato, On the Navier-Stokes initial value problem. I, Arch. Ration. Mech. Anal. 16 (1964), 269--315.

\bibitem[4]{Fujiwara}D. Fujiwara and H. Morimoto, An $L_{r}$-theorem of the Helmholtz decomposition of vector fields, J. Fac. Sci. Univ. Tokyo Sect. IA, Math. 24 (1977), 685--700.

\bibitem[5]{Giga 1}Y. Giga, Analyticity of the semigroup generated by the Stokes operator in $L_{r}$ spaces, Math. Z. 178 (1981), 297--329.

\bibitem[6]{Giga 2}Y. Giga, Domains of fractional powers of the Stokes operator in $L_{r}$ spaces, Arch. Ration. Mech. Anal. 89 (1985), 251--265.

\bibitem[7]{Giga 3}Y. Giga and T. Miyakawa, Solutions in $L_{r}$ of the Navier-Stokes initial value problem, Arch. Ration. Mech. Anal. 89 (1985), 267--281.

\bibitem[8]{Henry}D. Henry, Geometric Theory of Semilinear Parabolic Equations (Lecture Notes in Mathematics 840), Springer-Verlag, 1981.

\bibitem[9]{Hishida}T. Hishida, Existence and regularizing properties of solutions for the nonstationary convection problem, Funkcial. Ekvac. 34 (1991), 449--474.

\bibitem[10]{Kagei 1}Y. Kagei and M. Skowron, Nonstationary flows of nonsymmetric fluids with thermal convection, Hiroshima Math. J. 23 (1993), 343--363.

\bibitem[11]{Kagei 2}Y. Kagei, Attractors for two-dimensional equations of thermal convection in the presence of the dissipation function, Hiroshima Math. J. 25 (1995), 251--311.

\bibitem[12]{Lamb}H. Lamb, Hydrodynamics (Sixth Edition), Cambridge University Press, 1932.

\bibitem[13]{Lukaszewicz}G. {\L}ukaszewicz and P. Krzy\.{z}anowski, On the heat convection equations with dissipation term in regions with moving boundaries, Math. Methods Appl. Sci. 20 (1997), 347--368.

\bibitem[14]{Pazy}A. Pazy, Semigroups of Linear Operators and Applications to Partial Differential Equations (Applied Mathematical Sciences 44), Springer-Verlag, 1983.

\bibitem[15]{Serrin}J. Serrin, Mathematical Principles of Classical Fluid Mechanics (Fluid Dynamics I, Encyclopedia of Physics VIII/1), Springer-Verlag, 1959.

\bibitem[16]{Shibata}Y. Shibata and R. Shimada, On a generalized resolvent estimate for the Stokes system with Robin boundary condition, J. Math. Soc. Japan 59 (2007), 469--519.

\end{thebibliography}
\end{document}